\documentclass[11pt]{amsart}

\usepackage{amsmath,amssymb}
\usepackage{accents}
\usepackage{graphicx}
\usepackage{verbatim}
\usepackage{enumitem}
\usepackage{tikz}
\usepackage{caption}
\usepackage{subcaption}
\usepackage{tabu}
\usepackage{longtable}
\usepackage{array}
\usepackage{hyperref}

\newlength{\dhatheight}

\theoremstyle{definition}

\newtheorem{theorem}{Theorem}[section]
\newtheorem{conjecture}[theorem]{Conjecture}
\newtheorem{proposition}[theorem]{Proposition}
\newtheorem{lemma}[theorem]{Lemma}

\newtheorem{definition}[theorem]{Definition}

\newtheorem{fact}[theorem]{Fact}
\newtheorem{question}[theorem]{Question}

\newcommand{\bul}{\bullet}

\newlist{pcases}{enumerate}{1}
\setlist[pcases]{
  label={\em{Case~\arabic*:}}\protect\thiscase.~,
  ref=\arabic*,
  align=left,
  labelsep=0pt,
  leftmargin=0pt,
  labelwidth=0pt,
  parsep=0pt
}
\newcommand{\case}[1][]{%
  \if\relax\detokenize{#1}\relax
    \def\thiscase{}%
  \else
    \def\thiscase{~#1}%
  \fi
  \item
}

\begin{document}
\title[Straight Knots]
{Straight Knots}

\author[N.~Owad]{Nicholas Owad}
\address{Topology and Geometry of Manifolds Unit\\
Okinawa Institute of Science and Technology Graduate University\\
Okinawa, Japan 904-0495}
\email{nicholas.owad@oist.jp}

\begin{abstract}
Jablan and Radovi\'c originally defined two invariants called the Meander number and OGC number of knots for certain classes of knots.  We generalize these definitions to all knots and name the straight number and contained straight number of a knot, respectively, and prove they are well defined.   We answer two questions and prove a generalization of a conjecture of Jablan and Radovi\'c.  We also give some relations to crossing number and petal number.  Then we compute the straight numbers for all the knots in the standard knot table and present some interesting questions and the complete table of knots with 10 or fewer crossing and their straight number and contained straight number.
\end{abstract}

\maketitle



\section{Introduction}\label{sec:intro}


Knot diagrams are most commonly drawn with the minimum number of crossings.  This is how they appear in the knot table in Rolfsen \cite{Rolfsen} which is often referred to as the standard knot table.  Other common ways of presenting knots are with braids closures, in bridge position, thin position, and numerous others.  Recently, Adams et al. introduced \"ubercrossing and petal diagrams \cite{Multi}.  From most of these presentations of diagrams, invariants are created which are interesting in their own respect.   They are also useful for relations to other invariants and can help us understand different properties of knots.  In this paper, we introduce a new presentation for a knot diagram, called straight position and two new invariants, the straight number and the contained straight number.  These invariants are very closely related to the {\em meander number} and {\em OGC} defined by Jablan and Radovic \cite{JR}.

Gauss is said to have known the following fact, which is presented by Adams, Shinjo, and Tanka in \cite{AST}.

\begin{theorem}\label{thm:AST}{\cite[Theorem 1.2]{AST}}
Every knot has a projection that can be decomposed into two sub-arcs such that each sub-arc never crosses itself.
\end{theorem}

From this result, via a planar isotopy, one can produce a diagram with a single straight strand that contains all of the crossings. By convention, we will draw this straight arc horizontally.

\begin{figure}
\begin{center}
\begin{tikzpicture}[scale=.4]


\begin{scope}[xshift = -10cm]
\draw [ultra thick,gray] (-.2,.-.2) to [out=40, in=-90] (1.5,2.4);

\draw [ultra thick, gray] (-.2,-.2) to [out=180+40, in=90] (-1.2,-1.6) to [out=-90, in=180-30] (0,-3.2);

\draw [line width=0.13cm, white] (0,-3.2) to [out=-30, in=-90] (3.5,-.6) to [out=90, in=-20] (1.6 ,2) to [out=160, in=20] (-1.3 ,2.1);
\draw [ultra thick, gray] (0,-3.2) to [out=-30, in=-90] (3.5,-.6) to [out=90, in=-20] (1.6 ,2) to [out=160, in=20] (-1.3 ,2.1);

\draw [ultra thick, gray] (.2,-.2) to [out=-40, in=120] (1.05,-1);
\draw [ultra thick] (1.05,-1) to [out=300, in=70] (1.1,-2.3);
\draw [line width=0.13cm, white] (1.1,-2.3) to [out=-110, in=30] (0,-3.2) to [out=210, in=-90] (-3.5,-.6) to [out=90, in=200] (-1.6 ,2);
\draw [ultra thick, gray] (1.1,-2.3) to [out=-110, in=30] (0,-3.2) to [out=210, in=-90] (-3.5,-.6) to [out=90, in=200] (-1.6 ,2);

\path [fill=black] (1.05,-1) circle [radius=0.18];
\path [fill=black] (1.1,-2.3) circle [radius=0.18];

\draw [line width=0.13cm, white] (-.2,.2) to [out=140, in=-90] (-1.5,2.4);
\draw [ultra thick, gray] (-.2,.2) to [out=140, in=-90] (-1.5,2.4);

\draw [ultra thick, gray] (-1.5,2.4) arc [radius=1.5, start angle=180, end angle= 0];
\end{scope}


\draw [line width=0.1cm, ->] (4,0) -- (6,0);
\draw [line width=0.1cm, ->] (14,0) -- (16,0);
\draw [line width=0.1cm, ->] (-6,0) -- (-4,0);

\draw [ultra thick] (-.2,.-.2) to [out=40, in=-90] (1.5,2.4);

\path [fill=white] (1.5 ,2.04) circle [radius=0.18];

\draw [ultra thick, gray] (-.2,-.2) to [out=180+40, in=90] (-1.2,-1.6) to [out=-90, in=180-30] (0,-3.2);

\draw [line width=0.13cm, white] (0,-3.2) to [out=-30, in=-90] (3.5,-.6) to [out=90, in=-20] (1.6 ,2) to [out=160, in=20] (-1.3 ,2.1);
\draw [ultra thick, gray] (0,-3.2) to [out=-30, in=-90] (3.5,-.6) to [out=90, in=-20] (1.6 ,2) to [out=160, in=20] (-1.3 ,2.1);
\path [fill=white] (0 ,-3.2) circle [radius=0.18];

\draw [ultra thick] (-1,-1.5) to [out=-10, in=210] (.5,-.3) to [out=-40, in=120] (1.05,-1);
\draw [ultra thick] (1.05,-1) to [out=300, in=70] (1.1,-2.3);
\draw [line width=0.13cm, white] (1.1,-2.3) to [out=-110, in=30] (0,-3.2) to [out=210, in=-90] (-3.5,-.6) to [out=90, in=200] (-1.6 ,2);
\draw [ultra thick, gray] (1.1,-2.3) to [out=-110, in=30] (0,-3.2) to [out=210, in=-90] (-3.5,-.6) to [out=90, in=200] (-1.6 ,2);

\path [fill=black] (-.2,-.2) circle [radius=0.18];
\path [fill=black] (1.1,-2.3) circle [radius=0.18];

\draw [ultra thick] (-1.5,2.4) arc [radius=1.5, start angle=180, end angle= 0];

\draw [line width=0.13cm, white]  (-1.4,-1.3) to [out=140, in=225] (-.5,.4) to [out=140, in=-90] (-1.5,2.4);
\draw [ultra thick]  (-1.4,-1.3) to [out=140, in=225] (-.5,.4) to [out=140, in=-90] (-1.5,2.4);

\begin{scope}[xshift = 10cm]

\draw [ultra thick] (1.6,-3.2) to [out=90, in=0] (.5,-2.7) to [out=180, in=-90] (-.5,-1.7) to [out=90, in=215] (.5,-.3) to [out=-40, in=120] (1.05,-1);


\draw [ultra thick] (-.2,.-.2) to [out=40, in=-90] (1.5,2.4);



%
\draw [ultra thick] (-.2,-.2) to [out=180+40, in=90] (-1.2,-1.6) to [out=-90, in=180-30] (0,-3.2);

\draw [ultra thick] (0,-3.2) to [out=-30, in=180] (.8,-3.4);


%
\draw [ultra thick] (1.05,-1) to [out=300, in=70] (1.1,-2.3);

\draw [ultra thick]  (1.6,-3.2) to [out=-90, in=0] (.7,-3.9) to [out=180, in=-40]  (-1.1,-3.2) to [out=140, in=-90]   (-1.8,-1.7) to [out=90, in=225] (-.5,.4) to [out=140, in=-90] (-1.5,2.4);

\draw [line width=0.13cm, white] (.8,-3.4) to [out=0, in=-90] (3.5,-.6) to [out=90, in=-20] (1.6 ,2) to [out=160, in=20] (-1.3 ,2.1);

\draw [line width=0.13cm, white] (1.1,-2.3) to [out=-110, in=30] (0,-3.2) to [out=210, in=-90] (-3.5,-.6) to [out=90, in=200] (-1.7 ,2);

\draw [ultra thick, gray] (1.1,-2.3) to [out=-110, in=30] (0,-3.2) to [out=210, in=-90] (-3.5,-.6) to [out=90, in=200] (-1.7 ,2);
\draw [ultra thick, gray] (.8,-3.4) to [out=0, in=-90] (3.5,-.6) to [out=90, in=-20] (1.6 ,2) to [out=160, in=20] (-1.3 ,2.1);

\draw [ultra thick] (-1.5,2.4) arc [radius=1.5, start angle=180, end angle= 0];

\draw [line width=0.13cm, white]     (-1.8,-1.7) to [out=90, in=225] (-.5,.4) to [out=140, in=-90] (-1.5,2.4);
\draw [ultra thick]   (-1.8,-1.7) to [out=90, in=225] (-.5,.4) to [out=140, in=-90] (-1.5,2.4);

\path [fill=black] (.8,-3.4) circle [radius=0.18];
\path [fill=black] (1.1,-2.3) circle [radius=0.18];

\end{scope}
\begin{scope}[xshift = 19cm, yshift=1cm, scale=.8]

\draw [ultra thick] (-3,0) arc [radius=3, start angle=180, end angle= 0];
\draw [ultra thick] (-3,0) arc [radius=11/2, start angle=180, end angle= 360];
\draw [ultra thick] (-2,0) arc [radius=2, start angle=180, end angle= 0];
\draw [ultra thick] (-2,0) arc [radius=9/2, start angle=180, end angle= 360];
\draw [ultra thick] (-1,0) arc [radius=1, start angle=180, end angle= 0];
\draw [ultra thick] (-1,0) arc [radius=7/2, start angle=180, end angle= 360];

\draw [ultra thick] (0,0) arc [radius=1/2, start angle=180, end angle= 360];

\draw [ultra thick] (2,0) arc [radius=3/2, start angle=180, end angle= 360];

\draw [ultra thick] (3,0) arc [radius=1/2, start angle=180, end angle= 360];

\draw [ultra thick] (4,0) arc [radius=1/2, start angle=180, end angle= 0];

\draw [ultra thick] (6,0) arc [radius=1, start angle=180, end angle= 0];

\draw [line width=0.13cm, white] (0.5,0) to (3.85,0);
\draw [line width=0.13cm, white] (4.15,0) to (6.5,0);

\draw [ultra thick, gray] (0,0) to (3.8,0);
\draw [ultra thick, gray] (4.2,0) to (7,0);

\path [fill] (0,0) circle [radius=0.07];
\path [fill] (7,0) circle [radius=0.07];

\end{scope}

\end{tikzpicture}
\end{center}
\caption{The Figure 8 knot, being drawn so that the grey and the black arc never cross themselves. The last diagram is obtained by bending the grey arc so it is straight, pulling the black arc and all crossings with it.  The first 3 diagrams of this figure are from \cite{AST}.}\label{fig:AST}
\end{figure}
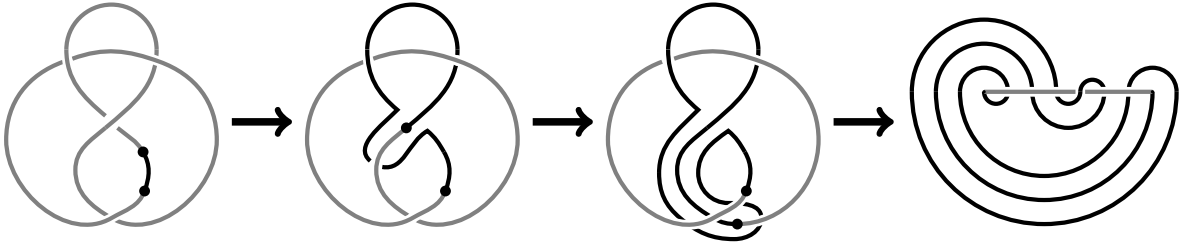


For precise definitions of the following terms, see Section \ref{sec:notation}. Once a diagram with $n$ crossings is in this straight position, we start at the left of the straight strand and move right, labeling the crossings 1 through $n$.  Continuing from the right end of the straight strand, list the crossing numbers we encounter if we were to continue traversing the knot.  Upon reaching the left end of the straight strand, we will have an $n$-tuple, with each number 1 to $n$ appearing once. We call this list the {\em straight word} of the diagram. The minimum number of crossings possible for a knot $K$ to have while in straight position is the {\em straight number} of the knot, $\mathtt{str} (K)$.  We produce a variation on straight knots, called contained straight knots, where we only look at diagrams which do not have arcs that pass around the side of the straight strand.  Again, the minimum number of crossings a knot diagram can have in contained straight position is the {\em contained straight number}, $\mathtt{cstr} (K)$.

\noindent{\bf{Theorem}~\ref{thm:cstr}.} 
{
The invariant, contained straight number,  $\mathtt{cstr}(K)$, is well defined.
}

The following question appears in Jablan and Radovi\'c \cite{JR}.

\begin{question} [\cite{JR}]
 Is it true that every (alternating) knot has a meander diagram?
\end{question}

 Theorem \ref{thm:cstr} answers this question  and, indeed, shows it is well defined for all knots, not just alternating knots.

A common question in knot theory is how invariants relate to each other.  Here, we present a bound on straight number, from the crossing number.

\noindent{\bf{Theorem}~\ref{thm:strANDcross}.} 
{
If $K$ is a knot with crossing number $c(K)=n$, then $\mathtt{str} (K)\leq 2^{n-1}-1$. 
}

We also relate the straight numbers and the petal number of a knot, $p(K)$.

\noindent{\bf{Theorem}~\ref{thm:strANDpetal}.}{ \begin{enumerate}
\item Given a knot $K$ with $\mathtt{str}(K) = n$, then $p(K) \leq 3n$. 
\item Given a knot $K$ with $\mathtt{cstr}(K) = n$, then $p(K) \leq 2n+3$. 
\end{enumerate}
}

As with Gauss Codes and DT codes, we must consider when these straight words produce classical knots or virtual knots.  

\noindent{\bf Theorem~\ref{thm:strExistsAlg}.} 
{
There exists an algorithm which can detect if a straight code is realizable as a classical knot.
}

We implement this algorithm in a python script and use the program SnapPy, \cite{snappy} by Culler, Dunfield, Goerner, and Weeks to identify the knots we obtain this way.  From this, we have calculated the straight number and contained straight number of the standard table of knots with 10 or fewer crossings.  These results are in a table at the end of the paper.

We also have the following results.
\newpage

\noindent{\bf Theorem~\ref{thm:4things}.}{
\begin{enumerate}
\item Every torus knot $T_{2,q}$ is perfectly straight.
\item Every $n$-pretzel knot is perfectly straight.
\item Every 2-bridge knot $K_{p/q}$ where the continued fraction of $p/q$ has length less than 6 is perfectly straight.  
\item Every knot with $7$ or less crossings is perfectly straight. 
\end{enumerate}
}

In the next section, we introduce relevant definitions.  In section 3, we find some relations to other invariants and in section 4, we discuss the calculations and algorithm mentioned in the previous paragraph.

{\bf Acknowledgements} The author would like to thank the staff at the high performance computing clusters at OIST for their help in troubleshooting the code used to produce the table at the end of this paper.  Also, the author wishes to thank Chaim Even Zohar for the suggestion to look at Jablan and Radovic's paper, \cite{JR}.

\smallskip


\section{Definitions and Well-definedness } \label{sec:notation}

We assume the reader is familiar with the basics of topology and knot theory, see \cite{Hatcher} and \cite{Rolfsen} for background.  Given a knot diagram, pick a point on the knot, which is not a crossing, pick a direction and traverse along the knot.  As we come to the first crossing, we label it `1' and continue on, labeling each crossing we encounter with the next consecutive number.  If a crossing has already been labeled, we skip it and move on.  This will result in an $n$ crossing knot with each crossing labeled once by the numbers 1 through $n$.   Listing the numbers of the crossings in order as we traverse the knot produces a $2n$-tuple.

\begin{definition} \label{def:knotword}
The {\em knot word} of a diagram is the $2n$-tuple we obtain by following the above directions.
\end{definition}

A quick note about knot words: For an $n$ crossing diagram, there are $4n$ knot words, which we obtain by picking one of the $2n$ segments of arcs between the crossings and then one of the 2 directions to travel.  These are related by a cyclic permutation or by reversing the order.

Using the definition and Theorem \ref{thm:AST}, there is a reasonable choice of starting position and direction at the beginning of one of the two colored arcs.  

\begin{definition} \label{def:strword}
A knot diagram is {\em straight} if it has a straight word of the form $$(1,2,3,\ldots,n-1,n, \sigma(1),\sigma(2),\sigma(3),\ldots,\sigma(n-1),\sigma(n))$$
 where $\sigma$ is an element of the symmetric group on $n$ elements, $S_n$. The {\em straight word} of a straight diagram is the $n$-tuple $$(\sigma(1),\sigma(2),\sigma(3),\ldots,\sigma(n-1),\sigma(n)).$$ 

\end{definition}

Note that in \cite{JR}, they refer to straight words as ``short Gauss Codes" or  ``short OGC Gauss Codes." The process required to make a knot straight can increases the number of crossings in the diagram.  Thus, we make a new invariant which captures this distinction.  

\begin{definition} \label{def:strnum}
The {\em straight number} of a knot $K$, $\mathtt{str}(K)$, is the minimum $n$ such that $K$ has an $n$ crossing straight diagram.
\end{definition}

This definition is well defined by Theorem \ref{thm:AST}, that is, every knot has a diagram that is straight.  Jablan and Radovi\'c make the following conjecture. 

\begin{conjecture}[\cite{JR}]\label{conj:JR}
Every rational knot is OGC knot.
\end{conjecture}

Definition \ref{def:strnum} being well defined proves Conjecture \ref{conj:JR}, and indeed every knot is an ``OGC" knot.

Then a knot is said to be in {\em straight position} if we draw the diagram so that the crossings all occur on a single horizontal strand.  This makes the second subarc consist of only semicircles, with their centers on the straight strand.  Thus, we have the following.

\begin{theorem}
Every knot has a diagram that is composed of a single straight strand and collection of semicircles, all with their centers on the straight strand.
\end{theorem}

We can make another definition by considering the number of times our knot must pass from the top of the diagram to the bottom but does not cross the straight strand.  To do this, the arc must cross the {\em extended straight strand}, which is obtained in the diagram by continuing the horizontal segment that is the straight strand, though not including it. See Figure \ref{fig:extendedSS}.

\begin{figure}
\begin{center}
\begin{tikzpicture}[scale=.4]

\draw [ultra thick] (-3,0) arc [radius=3, start angle=180, end angle= 0];
\draw [ultra thick] (-3,0) arc [radius=11/2, start angle=180, end angle= 360];
\draw [ultra thick] (-2,0) arc [radius=2, start angle=180, end angle= 0];
\draw [ultra thick] (-2,0) arc [radius=9/2, start angle=180, end angle= 360];
\draw [ultra thick] (-1,0) arc [radius=1, start angle=180, end angle= 0];
\draw [ultra thick] (-1,0) arc [radius=7/2, start angle=180, end angle= 360];

\draw [ultra thick] (0,0) arc [radius=1/2, start angle=180, end angle= 360];

\draw [ultra thick] (2,0) arc [radius=3/2, start angle=180, end angle= 360];

\draw [ultra thick] (3,0) arc [radius=1/2, start angle=180, end angle= 360];

\draw [ultra thick] (4,0) arc [radius=1/2, start angle=180, end angle= 0];

\draw [ultra thick] (6,0) arc [radius=1, start angle=180, end angle= 0];

\draw [line width=0.13cm, white] (0.5,0) to (3.85,0);
\draw [line width=0.13cm, white] (4.15,0) to (6.5,0);

\draw [ultra thick] (0,0) to (3.85,0);

\draw [ultra thick] (4.15,0) to (7,0);
\path [fill] (0,0) circle [radius=0.055];
\path [fill] (7,0) circle [radius=0.055];

\draw [ultra thick, dashed, gray] (-0.2,0) -- (-4,0);
\draw [ultra thick, dashed, gray] (7.2,0) -- (11,0);

\end{tikzpicture}
\end{center}
\caption{The grey dahsed line is called the {\em extended straight strand}.}\label{fig:extendedSS}
\end{figure}
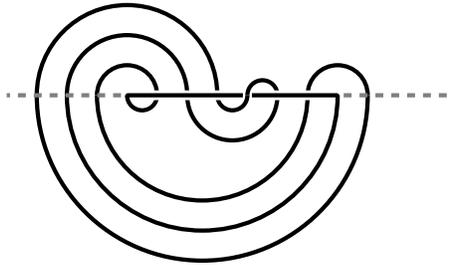

\begin{definition} \label{def:containedstrnum}
The {\em contained straight number} of a knot $K$, $\mathtt{cstr}(K)$, is the minimum $n$ such that $K$ has an $n$ crossing straight diagram where no arcs cross the extended straight strand.
\end{definition}

It is somewhat surprising that every knot has a contained straight position diagram.  But first, we notice a useful fact. For this argument, an arc will be a segment of the knot between two consecutive crossings.  Given a knot, find a diagram in straight position.  Assume there are arcs crossing the extended straight strand.  Call these arcs {\em uncontained arcs} and the other arcs call {\em contained arcs}.  Let $u$ and $c$ be the number of each type of these arcs in the diagram, respectively.  Note that any arc that crosses the extended straight strand twice does not have to cross it at all.  Reposition the other uncontained arcs so they all cross the left side of the extended straight strand. The last two sentences can be seen as isotopy of the diagram on the sphere, see Figure \ref{fig:twoTOone} and \ref{fig:allONleft}. Then every uncontained arc is made up of two semicircles and the contained arcs are a single semicircle.  Every contained and uncontained arc contributes exactly one crossing to the diagram, except the last arc, which is guaranteed to be contained, thus there are $c+u-1$ crossings in this diagram.  

\begin{fact}\label{uandc}
For a diagram in straight position with $u$ uncontained arcs and $c$ contained arcs, there are $c+2u$ semicircles and $c+u-1$ crossings.
\end{fact}

\begin{theorem}\label{thm:cstr}
The invariant, contained straight number,  $\mathtt{cstr}(K)$, is well defined.
\end{theorem}

\begin{proof}
Given a diagram in straight position, the left end of the straight strand connects to a contained arc.  Cut the knot at the left end of the straight strand  and extend the straight strand to the left crossing all the uncontained arcs.  Make all these new crossings over (or under) crossings.  Call this cut arc $x$ and shrink it slightly.   Along with Fact \ref{uandc},  notice that there are now $c+2u-1$ crossings and $c+2u-1$ semicircles.  We must now reconnect the knot with only over (or under) crossings to not change the type of knot.

Make a graph, $G$, by placing a vertex in each region of the diagram.  There is one region for each semicircle and one more for the unbounded region, $c+2u$.  Add an edge to the graph between two vertices if the regions are adjacent through a segment of the straight strand.  There will be one edge added to the right of each crossing, $c+2u-1$.  

{\em Claim: } $G$ is connected.  

If it were not connected, then $G$ would have a nontrivial subgraph $H$ that is not connected to $G - H$.  Consider $H$, it must have a vertex in a region that is adjacent to a region with a vertex in $G-H$, through either the straight strand or a contained or uncontained arc.  If it were adjacent by a segment of the straight strand it would not be disconnected, a contradiction.  If it were adjacent through only contained or uncontained arcs, then there must be a closed loop of only contained or uncontained arcs.  This would mean we have a link, which is a contradiction, and proves the claim.

 A well known theorem from graph theory is that any graph which is connected and has one less edge than it does vertices, is a tree.  Therefore, there is a unique path from the region where the cut arc $x$ is to the unbounded region that passes only through the straight strand. Create this path starting at the end of $x$.  We connect the path to the left side of the straight strand, and make all new crossings over (or under) crossings, finishing the proof.\end{proof}


\begin{figure}
\centering
\begin{minipage}{.5\textwidth}
  \centering
\begin{tikzpicture}[scale=.3]

\draw [ultra thick] (0,0) -- (7,0);
\draw [ultra thick, dashed, gray] (0,0) -- (-2,0);
\draw [ultra thick, dashed, gray] (7,0) -- (9,0);

\draw [ultra thick] (-1,0) arc [radius=1, start angle=180, end angle= 0];
\draw [ultra thick] (5,0) arc [radius=1.5, start angle=180, end angle= 0];
\draw [ultra thick] (-1,0) arc [radius=4.5, start angle=180, end angle= 360];

\draw [ultra thick, dotted] (1,0) arc [radius=2, start angle=180, end angle= 0];

\end{tikzpicture}

  \captionof{figure}{Uncontained arcs only cross the extended straight strand once.  Replace the solid arc with the dotted arc. }
  \label{fig:twoTOone}
\end{minipage}%
\begin{minipage}{.5\textwidth}
  \centering
\begin{tikzpicture}[scale=.3]

\draw [ultra thick] (0,0) -- (7,0);
\draw [ultra thick, dashed, gray] (0,0) -- (-3,0);
\draw [ultra thick, dashed, gray] (7,0) -- (9,0);

\draw [ultra thick] (-1,0) arc [radius=1, start angle=180, end angle= 0];
\draw [ultra thick] (-1,0) arc [radius=2, start angle=180, end angle= 360];

\draw [ultra thick] (5,0) arc [radius=1.5, start angle=180, end angle= 0];
\draw [ultra thick] (5.5,0) arc [radius=1.25, start angle=180, end angle= 360];

\draw [ultra thick, dotted] (-2,0) arc [radius=3.5, start angle=180, end angle= 0];
\draw [ultra thick, dotted] (-2,0) arc [radius=3.75, start angle=180, end angle= 360];

\end{tikzpicture}
  \captionof{figure}{All uncontained arcs can be modified to cross the left side of the extended straight strand. Replace the rightmost arc with the dotted arc.}
  \label{fig:allONleft}
\end{minipage}
\end{figure}

It is clear that for any knot $K$, 
$$c(K) \leq \mathtt{str} (K) \leq \mathtt{cstr} (K),$$
where $c(K)$ is the minimum crossing number.  

\begin{definition}
A knot is {\em perfectly straight} if $c(K) = \mathtt{str} (K)$.
\end{definition}

A knot which has a minimum crossing diagram where, traversing the knot, we can meet every crossing before coming to a crossing for the second time, is perfectly straight.  Now, we discuss how many arcs can be uncontained in a straight diagram with $n$ crossings.
\newpage 
\begin{lemma}\label{lem:uncontained}
Given a nontrivial knot $K$ with $\mathtt{str} (K)=n$, there are at most $n-3$ uncontained arcs and at least four contained arcs. 
\end{lemma}

\begin{proof}
By convention, we draw our straight strand from left to right and the first semicircle above the straight strand.  This arc may be assumed to be contained, see Figure \ref{fig:firstArc}.  We could merely draw the first semicircle on the bottom first and then flip the diagram over the straight strand. In a straight diagram with $n$ crossings, if the first crossing in the straight word is $n$, it is a nugatory crossing and can be removed.  The crossing cannot be $n-1$ since this will produce a non-realizable straight word, see Figure \ref{fig:containedSingle}.  Thus, the largest the first entry in the straight word can be is $n-2$.   Hence, the first semi-circle is a contained arc and any arcs inside the first semi-circle must also be contained.  There is at least one arc inside the first semicircle.  Through a similar argument, we have two contained arcs at the beginning of the straight strand, but these may be on the top or on the bottom.   Thus, we have shown that there are at least four contained arcs. Let $u$ be the number of uncontained arcs and $c$ the number of contained arcs.  By Fact \ref{uandc}, $u+c-1 =n$ and thus $c = n +1 - u.$  Since $c\geq 4$, we have  $ n+ 1 - u \geq 4$ and therefore $u \leq n-3$.
\end{proof}

Note that this bound is sharp.  In fact, the figure-8 knot achieves this and generalizes for every $n= 4+3i$, for any $i\geq 0$.  See Figure \ref{fig:AcheiveBound} for examples of knots with $i=0$ and $i=1$.


\begin{figure}
\centering
\begin{minipage}{.5\textwidth}
  \centering
\begin{tikzpicture}[scale=.3]

\draw [ultra thick] (0,0) -- (7,0);
\draw [ultra thick, dashed, gray] (0,0) -- (-2,0);

\draw [ultra thick] (-1,0) arc [radius=4, start angle=180, end angle= 0];
\draw [ultra thick] (-1,0) arc [radius=1.5, start angle=180, end angle= 360];
\draw [ultra thick, dotted] (2,0) arc [radius=2.5, start angle=180, end angle= 360];
\path [fill] (7,0) circle [radius=0.06];

\end{tikzpicture}

  \captionof{figure}{The first arc never needs to be uncontained. }
  \label{fig:firstArc}
\end{minipage}%
\begin{minipage}{.5\textwidth}
  \centering
\begin{tikzpicture}[scale=.4]

\draw [ultra thick] (0,0) -- (7,0);
\draw [ultra thick, dashed] (0,0) -- (-2,0);

\draw [ultra thick] (3,0) arc [radius=2, start angle=180, end angle= 0];
\draw [ultra thick, dashed] (3,0) to (3,-1);

\draw [ultra thick, dashed] (5,.75) to (5,-.75);

\path [fill] (7,0) circle [radius=0.06];

\node [ ] at (5,-1) {$n$};
\node [ ] at (3,-1) {$n-1$};

\end{tikzpicture}
  \captionof{figure}{The first arc cannot enclose a single crossing. }
  \label{fig:containedSingle}
\end{minipage}
\end{figure}

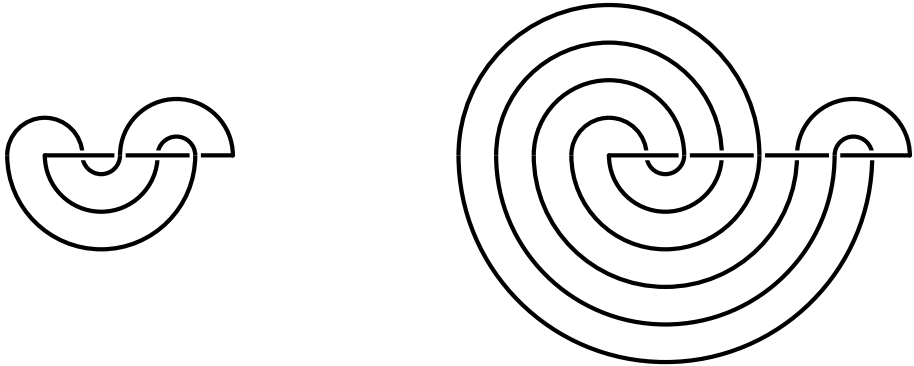
\begin{figure}
  \centering
\begin{tikzpicture}[scale=.5]

\draw [ultra thick] (2,0) arc [radius=1.5, start angle=180, end angle= 0];
\draw [ultra thick] (3,0) arc [radius=.5, start angle=180, end angle= 0];
\draw [ultra thick] (-1,0) arc [radius=1, start angle=180, end angle= 0];
\draw [ultra thick] (-1,0) arc [radius=2.5, start angle=180, end angle= 360];
\draw [ultra thick] (0,0) arc [radius=1.5, start angle=180, end angle= 360];
\draw [ultra thick] (1,0) arc [radius=.5, start angle=180, end angle= 360];

\draw [line width=0.13cm, white] (0.5,0) to (1.5,0);
\draw [line width=0.13cm, white] (2.5,0) to (3.5,0);
\draw [ultra thick] (0,0) -- (1.85,0);
\draw [ultra thick] (2.15,0) -- (3.85,0);
\draw [ultra thick] (4.15,0) -- (5,0);
\path [fill] (0,0) circle [radius=0.06];
\path [fill] (5,0) circle [radius=0.06];


\begin{scope}[xshift=15cm]
\draw [ultra thick] (5,0) arc [radius=1.5, start angle=180, end angle= 0];
\draw [ultra thick] (6,0) arc [radius=.5, start angle=180, end angle= 0];
\draw [ultra thick] (0,0) arc [radius=1.5, start angle=180, end angle= 360];
\draw [ultra thick] (1,0) arc [radius=.5, start angle=180, end angle= 360];
\draw [ultra thick] (-1,0) arc [radius=1, start angle=180, end angle= 0];
\draw [ultra thick] (-1,0) arc [radius=2.5, start angle=180, end angle= 360];
\draw [ultra thick] (-2,0) arc [radius=2, start angle=180, end angle= 0];
\draw [ultra thick] (-2,0) arc [radius=3.5, start angle=180, end angle= 360];
\draw [ultra thick] (-3,0) arc [radius=3, start angle=180, end angle= 0];
\draw [ultra thick] (-3,0) arc [radius=4.5, start angle=180, end angle= 360];
\draw [ultra thick] (-4,0) arc [radius=4, start angle=180, end angle= 0];
\draw [ultra thick] (-4,0) arc [radius=5.5, start angle=180, end angle= 360];

\draw [line width=0.13cm, white] (0.5,0) to (1.5,0);
\draw [line width=0.13cm, white] (2.5,0) to (3.5,0);
\draw [line width=0.13cm, white] (4.5,0) to (5.5,0);
\draw [line width=0.13cm, white] (6.5,0) to (7.5,0);
\draw [ultra thick] (0,0) -- (1.851,0);
\draw [ultra thick] (2.15,0) -- (3.85,0);
\draw [ultra thick] (4.15,0) -- (5.85,0);
\draw [ultra thick] (6.15,0) -- (8,0);
\path [fill] (0,0) circle [radius=0.06];
\path [fill] (8,0) circle [radius=0.06];
\end{scope}

\end{tikzpicture}
  \captionof{figure}{The Figure 8 knot and the knot $7_7$. Both of these diagrams have the most uncontained arcs possible for the number of crossings. }
  \label{fig:AcheiveBound}
\end{figure}

One should also ask how straight number behaves under connect sum.  Clearly, we have 
$$ \mathtt{str} (K\# J) \leq \mathtt{str} (K) +\mathtt{str} (J),$$
 and similarly for contained straight number, but we have no lower bound, but we can guess at one.  Indeed, straight position might have enough structure to prove what we have yet to show in crossing number.

\begin{conjecture}
Given any two knots $K, J$, $$ \mathtt{str} (K\# J) = \mathtt{str} (K) +\mathtt{str} (J)\text{ and } \mathtt{cstr} (K\# J) = \mathtt{cstr} (K) +\mathtt{cstr} (J).$$
\end{conjecture}


\section{Relation to other Invariants}\label{sec:invariants}


If a given knot is not perfectly straight, we would like to know how many new crossing we are forced to introduce to put the diagram in straight position. 

\begin{lemma}
Given a knot diagram, let $m$ be the maximum number of crossings we can find before we meet a crossing the second time and let $r$ be the remaining number of crossings.  Then $\mathtt{str} (K)\leq 2^r(m+1)-1$. 
\end{lemma}

\begin{proof}
For a given knot $K$, with $c(K)=n$, let $m$ be the maximum number of crossings in a row we can find that have no duplicates. Let $r=n-m$ be the remaining number of crossings.  Draw the diagram as straight as possible, which will yield $m$ straight crossings from left to right, then the remaining $r$ crossing are not on the straight strand.  From the end of the straight strand, our knot must cross some subset of the previous $m$ crossings before reaching the remainders.  If it did not, $m$ would not be the maximum number of crossings we can find before we meet a crossing the second time.  If the path to the first remainder crossing passes more than half of the $m$ straight crossings, relabel the diagram so that we start at the right of the straight strand and the $m$-th crossing is at the left.  Then we have to pass no more than half of the straight crossings to get to the new first remainder crossing. 

The first remainder crossing, $r_1$, has two arcs, let $x$ be the arc that we followed from the end of the straight strand, and $y$ be the other.  Note that $x$ crosses the straight strand at most $m/2$ times.  Perturb the arc $y$ and push it back, following along $x$ crossing the straight strand until the crossing is on the straight strand.  We introduce two new crossings each time we pass the straight strand and $r_1$ is now on the straight strand.  Therefore we now have at most $m+2\frac{m}{2}+1= 2m+1$ crossings on the straight strand and there are $r-1$ remaining crossings.  Repeating this process, but with the new number of straight crossings, we have $2(2m+1)+1=2^2m+2^2-1$ straight crossings and $r-2$ remaining.  Repeating $r-2$ more times gives us our result.
\end{proof}

Given any knot, we can see that there is a reduced diagram with $m\geq 3$, and thus $r=n-3$.  This gives us the following result.

\begin{theorem}\label{thm:strANDcross}
If $K$ is a knot with crossing number $c(K)=n$, then $\mathtt{str} (K)\leq 2^{n-1}-1$. 
\end{theorem}

Clearly, we can almost always find much higher values for $m$ for a specific diagram, but it seems challenging to increase $m$ in a significant way for all diagrams.   It would also be valuable to find a non exponential bound for straight number in terms of crossing number, if one exists.

\begin{question}
What is the upper bound for straight number in terms of the crossing number?
\end{question}

While it appears likely that straight number behaves well under connect sum, it seems very unlikely that it does so under the operation of creating satellites.  Imagine a  satellite of index $j>1$.  This will mean we have $j$ strands running along what used to be the straight strand, with $j$ strands crossing the $j$ straight strands.  To succeed in putting this knot into straight position, we expect a very high number of crossings that needs to be introduced.  It is unlikely that the bound in Theorem \ref{thm:strANDcross} is sharp, but we expect satellite knots are a class of knots which come close.

\begin{proposition}
If $K$ is a knot with $\mathtt{str} (K)=n$, then $\mathtt{cstr} (K)\leq 4n-8$. 
\end{proposition}

\begin{proof}
Consider $u$, the number of uncontained arcs and the number of crossings, $n$.  We know that $u\leq n-3$ and we make a new crossing for each uncontained arc and then there are no more than $2n-3$ crossings, by Lemma \ref{lem:uncontained}.  Reconnecting the knot can at worse double the number of crossings since we will only go between any two adjacent crossing at most once.  This gives us $4n-6$.  Consider the last crossing in the diagram before we extend the straight strand.  If it is an over (under) crossing, we make all the new crossing over (under), yielding a Type II Reidemeister move.     This will eliminate two crossings, giving our result. \end{proof}

In \cite{Multi}, Adams et al. define the petal projection of a knot and the associated invariant, the  petal number of a knot, $p(K)$.   Here we give some relations between petal number and straight number.

\begin{lemma}\label{lem:uberVSstr}
If $K$ is a knot in straight position with $u$ uncontained arcs and $c$ contained arcs, then $p(K) \leq 3u+2c+1$.
\end{lemma}

\begin{proof}
Given a diagram with $u$ uncontained arcs and $c$ contained arcs, let the diagram lie in the $x,y$-plane, where the straight strand corresponds to the $x$-axis.  We can perturb the straight strand slightly to keep all midpoints between every adjacent pair of crossings fixed and push the segments between the straight strand up into the positive $z$ direction if the straight strand was the overarc of the crossing, and down into the negative $z$ direction if the straight strand was the underarc.  If there are two adjacent semicircles in a row that are both in the positive or negative $z$ direction, combine them into a single semicircle. Repeat this last step until all semicircles we create from the straight strand alternate from the positive $z$ direction to the negative.  This process creates at most one new semicircle for each crossing.  Hence, $u+c-1$ new semicircles. 

Now, if we look down the $x$-axis at the $y,z$-plane, we would see a plus-shaped profile where the straight strand semicircles lie in the $z$ direction and the contained and uncontained arcs live in the $y$ direction.  We start with the straight strand's semicircles in the $z$ direction.  We let the first semicircle's edge stay stationary along the positive or negative $z$-axis, and rotate the rest of the semicircle to form a petal.  Then we pass through the $x$-axis to the diagonal quadrant in the $y,z$-plane and continue rotating all the semi circles slightly as we traverse the straight strand, until we have petals that fully fill the two diagonal quadrants.  This will create at most $u+c-1$ petals.  At the last semicircle of the straight strand, we will be tangent to the $y$-axis.  Either the first contained arc from this the end of the straight strand will continue in the same direction or the opposite direction.  If necessary, we add in a petal in the new quadrant, so they are compatible directions.  Then continue rotating all the semicircles of the uncontained and contained arcs.  This will fill the remaining two quadrants.  To finish the knot, we may need to add in another petal when we reach the last arc similar to the end of the straight strand.  This gives us at most $2u+c+2$ petals that we added in second pair of quadrants.  Hence, over all the quadrants, there are at most $3u+2c+1$ petals.
\end{proof}

The next theorem comes directly from applying Lemmae \ref{lem:uberVSstr} and \ref{lem:uncontained} and the definition of a contained straight knot.

\begin{theorem}\label{thm:strANDpetal}~
\begin{enumerate}
\item Given a knot $K$ with $\mathtt{str}(K) = n$, then $p(K) \leq 3n$. 
\item Given a knot $K$ with $\mathtt{cstr}(K) = n$, then $p(K) \leq 2n+3$. 
\end{enumerate}
\end{theorem}


\section{Algorithm and the Table of straight numbers}\label{sec:computation}

There are of course, elements of the symmetric group $S_n$ which do not give a straight word that is realizable.  Similar to Gauss codes and DT codes, many straight codes will produce virtual knots.  Before we give the main theorem of this section, we give a simpler version for contained straight knots.

\begin{theorem} \label{thm:cstrExistsAlg}
 Given a word of length $n$, there exists an algorithm which can detect if it is realizable as a contained straight knot in $\frac{n^2}{4}$ time.  
\end{theorem}

\begin{proof}
Let $w=(x_1,x_2,\ldots,x_n)$ be a list of the $n$ integers from 1 to $n$.  To determine if $w$ is a contained straight word, we must check that the word is realizable as a classical knot or a virtual knot.  Since we are assuming that there are $n$ crossings, from the length of $w$, the virtual crossings will only occur when two semicircles intersect.  We also assume that there are no uncontained arcs in our diagram.  Let us assume the diagram lives in the $x,y$-plane with the straight strand on the $x$-axis starting at $x=0$ and ending at $x=n+1$.  Then the crossings are at $x=i$ for $i\in\{1,2,\ldots,n\}$. Thus, the first semicircle, by convention is above the straight strand, has endpoints at $n+1$ and $x_1$, which we will write as $[n+1,x_1]$.  Then we have crossed the straight strand, and will make another semicircle on the bottom with endpoints $[x_1,x_2]$.  Continuing this until we have made all the semicircles, we will have a set of circles on top, $C_{top}$ and bottom, $C_{bot}$.  Then

$$C_{top} = \{ [n+1,x_1], [x_2,x_3], \ldots, [x_{2j},x_{2j+1}],\ldots  \},$$
$$C_{bot} = \{ [x_1,x_2], [x_3,x_4], \ldots, [x_{2j-1},x_{2j}],\ldots  \},$$
and the last semicircle $[x_n,0]$ will be on the top or bottom depending on the parity of $n$.  With this setup, all we must do is check whether the semicircles in $C_{top}$ intersect each other and similarly in $C_{bot}$.  Each pair of semicircles are either nested, adjacent, or intersecting.  The only one of these options that would imply our word is not contained straight knot is if the circles intersect.  Then using the cross ratio, $cr$, on the endpoints of the semicircles will will produce either a positive or negative number.  If the semicircles are intersecting, the cross ratio will be negative and otherwise the cross ratio is positive.  $$cr([x_1,x_2],[x_3,x_4]) = \frac{(x_3-x_1)(x_4-x_2)}{(x_3-x_2)(x_4-x_1)}$$
Therefore, we just check all the pairs in $C_{top}$ and $C_{bot}$ with the cross ratio and if we only have positive values, we have verified that $w$ is realizable as a contained straight knot. 

To see that the algorithm described above runs in quadratic time, note that there are $n+1$ semicircles and these are either on top or bottom, hence $\#C_{top} = \lceil \frac{n +1}{2} \rceil$ and $\#C_{bot} = \lfloor \frac{n +1}{2} \rfloor$. The number of times we need to use the cross ratio for a set of $m$ semicircles is ${ m\choose 2 }= \frac{m(m-1)}{2}$. Thus, in total, we need $$\frac{\lceil \frac{n +1}{2} \rceil(\lceil \frac{n +1}{2} \rceil-1)}{2}+ \frac{\lfloor \frac{n +1}{2} \rfloor(\lfloor \frac{n +1}{2} \rfloor-1)}{2}$$
instances of using the cross ratio, which reduces to $\frac{n^2-1}{4}$ for $n$ odd and $\frac{n^2}{4}$ for $n$ even.  \end{proof}

\begin{theorem} \label{thm:strExistsAlg}
There exists an algorithm which can detect if a straight code is realizable as a classical knot.
\end{theorem}

\begin{proof}
Let $w=(x_1,x_2,\ldots,x_n)$ be a list of the $n$ integers from 1 to $n$, where we ignore whether a crossing is over or under. First, we run the same algorithm from Theorem  \ref{thm:cstrExistsAlg}.  If it completes, we are done.  If the word is not contained straight, then it may have uncontained arcs.  There is at most $n-3$ uncontained arcs in a straight word of length $n$.  We denote these by inserting a negative number in our straight word, to signify going around the straight strand on the left.  This will let the uncontained arc which is two semicircles be distinguished as such in the word.  

For example, with the word $(2,1,4,3)$, we think of the first semicircle as $[5,2]$ coming from the end of the straight strand, which will intersect the semicircle $[1,4]$ if they are both contained.  So, we add a negative number to between 1 and 4 to signify crossing the extended straight strand.  The number we insert will always be the negative of the minimum of the previous crossing and the next crossing.  Thus, for this example, $-1$.  So we augment our word with this new number: $(2,1,-1,4,3)$.  Thus, we have the semicircle $[1,-1]$ on top and $[-1,4]$ on the bottom.  Now, we run the algorithm from Theorem \ref{thm:cstrExistsAlg} again. We repeat this process for every possible position and number of uncontained arcs in our word. One of these positions will pass this algorithm if and only if the word is realizable as a classical knot.   \end{proof}

Note, that we will never have a situation where we will augment the word with the same negative number at two points in our word, i.e., the negative number we augment our word with is completely determined by its neighbors.   For example, $(\ldots, 8, 3, 6\ldots)$ will never have need of two uncontained arcs at the positions marked by bullets, $(\ldots,8, \bul, 3, \bul, 6 \ldots)$.  One may verify this by an exhaustively checking all the possible situtations.  This is not to say that you cannot have two uncontained arcs appear in a row, see Figure \ref{fig:AcheiveBound}.  All this says is that if you do have consecutive uncontained arcs, the values in your word will either be decreasing or increasing.  This corresponds to spiraling in or out.   The word $(\ldots,8, \bul, 5, \bul, 2 \ldots)$ is a valid configuration and the augmented word would be  $(\ldots,8, -5, 5, -2, 2, \ldots)$.

Each of the parts of the following theorem are easily verified.

\begin{theorem}\label{thm:4things}
\begin{enumerate}
\item Every torus knot $T_{2,q}$ is perfectly straight.
\item Every $n$-pretzel knot is perfectly straight.
\item Every 2-bridge knot $K_{p/q}$ where the continued fraction of $p/q$ has length less than 6 is perfectly straight.  
\item Every knot with $7$ or less crossings is perfectly straight. 
\end{enumerate}
\end{theorem}

Proving a family of knots is perfectly straight is relies only on finding a straight diagram with the same number of crossings.  Proving knots are not perfectly straight is not as simple.

\begin{question}
Can we find families of knots which are not perfectly straight?
\end{question}

We have constructed a python script which runs with the program SnapPy.  This script enumerates all the possible straight words, with some obvious conditions to cut down the number of these words.  Then we fed them to SnapPy which identified them for us.  The words which resisted identification were few and were tackled from other directions.  The result is the following list of the straight number of all the knots in the standard Rolfsen table.  Note that here, the negatives in the straight words correspond to the contained and uncontained arcs being under the straight strand, contrary to what negatives represent in Theorem \ref{thm:strExistsAlg}.

Some interesting points about the results in Table \ref{table1}, we notice that there are gaps in contained straight number at even numbers.  For example, there are no knots with contained straight number 4 or 6, and there are only three knots with contained straight number 8.  Does this trend continue? Next are some more in depth questions.

\begin{question}~
\begin{enumerate}
\item Is it possible to characterize what makes a knot perfectly straight from geometric or topological perspectives?
\item Is it possible to characterize when a knot $K$ has $\mathtt{str}(K)<\mathtt{cstr}(K)$ from geometric or topological perspectives?
\end{enumerate}
\end{question}

Note that the largest number to appear in this table is $\mathtt{cstr}(10_{96})=\mathtt{cstr}(10_{123})=14$.  The knot $10_{123}$ can be obtained by taking the torus knot $T_{5,3}$, which happens to be $10_{124}$, and making it alternating.  Will this process of making large torus knots alternating produce hyperbolic knots which have the largest gaps between contained straight number and crossing number?  In Jablan and Radovic's paper \cite{JR}, they note that the only 10 crossing knots they failed to find meander diagrams for are these two knots and $10_{99}$, which has $\mathtt{cstr}(10_{99}) = 13$.  The contained straight words for these knots appear in the table below.  In our work to find all the values in this table, we used approximately 100 nodes on the  high performance computing clusters located at Okinawa Institute of Science and Technology Graduate University for one whole week.  This was the most intensive computation for this work.  Note that this computation also yielded many results for 12, 13, and 14 crossing knots, which we do no include here.  

The author has also made a javascript program which draws knots in straight position.   This runs in your web browser and the user can enter their own straight word to be drawn.  This can be found on the author's website. \url{http://nick.owad.org/drawstr.html}

{\scriptsize
\begin{longtable}{|c|c|c|c|c|}
\caption{Rolfsen's table of knots with straight number, contained straight number, and examples of words which attain each value. }\label{table1}\\
\hline
\textbf{$K$} & \textbf{$\mathtt{str}$} & \textbf{$\mathtt{cstr}$ } & \textbf{Straight Word} & \textbf{Contained Straight Word} \\
\hline
\endfirsthead
\multicolumn{5}{r}%
{\tablename\ \thetable\ -- \textit{Continued from previous page}} \\
\hline
\textbf{$K$} & \textbf{$\mathtt{str}$} & \textbf{$\mathtt{cstr}$ } & \textbf{Straight Word} & \textbf{Contained Straight Word} \\\hline
\endhead
\hline \multicolumn{5}{r}{\textit{Continued on next page}} \\
\endfoot
\hline
\endlastfoot
 $3_1$ & 3 & 3 & $(1,-2,3)$ & $ -$ \\ 
 \hline
 $4_1$ & 4 & 5 & $(2,-1,\bul,4,-3)$ & $(1, 4, -3, 2, 5)$\\
 \hline
 $5_1$ & 5 & 5 & $(1, -2, 3, -4, 5) $ & $ - $ \\
 \hline
 $5_2$ & 5 & 5 & $(1, -4, 3, -2, 5) $ & $ - $ \\
 \hline
  $6_1$ & 6 & 7 & $(2, -1, \bul,6, -5, 4, -3) $ & $  (1, -4, 3, -2, -7, 6, -5) $ \\
 \hline
  $6_2$ & 6 & 7 &  $(2, -1,\bul, 4, -5, 6, -3)$ & $ (1, -2, 3, 6, -5, 4, 7) $\\
 \hline
  $6_3$ & 6 &7  & $(2, -3, 6,\bul ,-1, 4, -5)$ & $ (1, -6, -3, 4, -5, -2, 7) $ \\
 \hline
  $7_1$ & 7 & 7  & $(1, -2, 3, -4, 5, -6, 7)$ & $ - $ \\
 \hline
  $7_2$ & 7 & 7 &$ (1, -6, 5, -4, 3, -2, 7)$ & $ - $ \\
 \hline
  $7_3$ & 7 & 7 & $(1, -2, 3, -6, 5, -4, 7)$ & $  -$ \\
 \hline
  $7_4$ & 7 & 7 & $(1, -4, 3, -2, 7, -6, 5) $ & $ -  $ \\
 \hline
  $7_5$ & 7 & 7  & $(1, -6, 3, -4, 5, -2, 7)$ & $ - $ \\
 \hline
  $7_6$ & 7 &  8 &  $(3, -2,\bul, 7, -6,\bul, 1, -4, 5)$ & $ (4, -5, -8, 1, -2, 7, -6, 3)  $\\
 \hline
  $7_7$ & 7 & 9 & $(3, -2,\bul, 7, -4, 1, \bul , -6, 5)$ & $ (1, 4, -3, 2, 5, 8, -7, 6, 9)  $ \\
 \hline
  $8_1$ & 8 & 9  & $(2, -1,\bul, 8, -7, 6, -5, 4, -3)$ & $ (1, 4, -3, 2, 9, -8, 7, -6, 5)  $ \\
 \hline
  $8_2$ & 8 & 9  & $(2, -1,\bul, 4, -5, 6, -7, 8, -3) $ & $ (1, -2, 3, -4, 5, 8, -7, 6, 9) $ \\
 \hline
  $8_3$ & 8 & 9 & $(4, -3, 2, -1,\bul, 8, -7, 6, -5) $ & $ (1, -4, 3, -2, -9, 8, -7, 6, -5) $ \\
 \hline
  $8_4$ & 8 & 9 & $ (2, -3, 4, -1,\bul, 8, -7, 6, -5) $ & $ (1, 6, -3, 4, -5, 2, 9, -8, 7)  $ \\
 \hline
  $8_5$ & 8 & 9 & $ (2, -1,\bul, 6, -7, 8, -3, 4, -5)$ & $ (1, 8, -5, 6, -7, 2, -3, 4, 9)  $ \\
 \hline
  $8_6$ & 8 & 9  & $(2, -1,\bul, 6, -7, 8, -5, 4, -3) $ & $(1, -6, 3, -4, 5, -2, -9, 8, -7)  $ \\
 \hline
  $8_7$ & 8 & 9  & $(2, -3, 4, -5, 8, \bul, -1, 6, -7) $ & $ (1, -2, 3, -8, -5, 6, -7, -4, 9) $ \\
 \hline
  $8_8$ & 8 & 9 &  $(2, -3, 8, -7, 6,\bul, -1, 4, -5) $ & $ (1, -8, -5, 6, -7, -4, 3, -2, 9) $\\
 \hline
  $8_9$ & 8 & 9 & $ (2, -3, 4, -1,\bul, 6, -7, 8, -5) $ & $ (1, -8, -5, 6, -7, -2, 3, -4, 9) $ \\
 \hline
   $8_{10}$ & 8 & 9 & $(2, -3, 6, -7, 8,\bul, -1, 4, -5) $ &  $ (1, -8, -5, 6, -7, -2, 3, -4, 9) $ \\
 \hline
  $8_{11}$ & 8 &  9&  $ (2, -1,\bul, 4, -7, 6, -5, 8, -3)$ &  $ (1, -4, 3, -2, 5, 8, -7, 6, 9)$\\
 \hline
  $8_{12}$ & 8 & 10 & $ (6, -5, 2, -1,\bul, 8, -7, 4, -3)$ &  $ (4, 7, -6, 5, 10, 1, -2, -9, 8, -3) $ \\
 \hline
  $8_{13}$ & 8 & 9 & $ (2, -5, 4, -3, 8,\bul, -1, 6, -7) $ &  $ (1, -8, -3, 6, -5, 4, -7, -2, 9) $ \\
 \hline
  $8_{14}$ & 8 & 9 &  $ (2, -3, 8, -7, 4, -1,\bul, 6, -5)$ &  $ (1, -8, 3, 6, -5, 4, 7, -2, 9) $ \\
 \hline
  $8_{15}$ & 8 & 8 & $ (4, -5, 8, -1, 2, -7, 6, -3) $ &  $ -$ \\
 \hline
  $8_{16}$ & 10 &  10 &  $ (4, -5, -8, 9, -10, 1, -2, -7, 6, 3)$ &  $ - $ \\
 \hline
  $8_{17}$ & 8 & 10 & $ (2, -5, 6, -1,\bul, 8, -3, 4, -7)$ &  $ (4, -5, -8, 9, -10, -1, 2, 7, -6, -3) $  \\
 \hline
  $8_{18}$ & 10 & 11 & $(2, -3, -6, 7, 10, \bul, -1, -4, 5, 8, -9) $ & $ (1, 4, -5, -10, 9, 6, 3, -2, -7, 8, 11) $ \\
 \hline
  $8_{19}$ & 8 & 9 & $(2, -1,\bul, -6, 7, -8, 3, -4, 5) $ &  $ (1, 8, 5, -6, 7, -2, 3, -4, 9)$ \\
 \hline
  $8_{20}$ & 8 & 8 & $(4, -5, 8, 1, -2, -7, 6, 3) $ &  $ -$ \\
 \hline
  $8_{21}$ & 8 & 9 & $(4, -5, 8, -1, 2, 7, -6, -3) $ &  $ - $ \\
 \hline
  $9_{1}$ & 9 & 9 & $ (1, -2, 3, -4, 5, -6, 7, -8, 9) $ &  $ - $ \\
 \hline
  $9_{2}$ & 9 & 9 & $ (1, -2, 9, -8, 7, -6, 5, -4, 3) $ &  $ - $ \\
 \hline
  $9_{3}$ & 9 & 9 & $ (1, -2, 3, -4, 5, -8, 7, -6, 9) $ &  $ - $ \\
 \hline
  $9_{4}$ & 9 & 9 & $ (1, -2, 3, -8, 7, -6, 5, -4, 9) $ &  $ - $ \\
 \hline
  $9_{5}$ & 9 & 9 & $ (1, -4, 3, -2, 9, -8, 7, -6, 5) $ &  $ - $ \\
 \hline
  $9_{6}$ & 9 & 9 & $ (1, -4, 5, -6, 7, -8, 3, -2, 9) $ &  $ - $ \\
 \hline
  $9_{7}$ & 9 & 9 & $ (1, -2, 7, -8, 9, -6, 5, -4, 3) $ &  $ - $ \\
 \hline
  $9_{8}$ & 9 & 10 & $ (3, -2,\bul, 9, -8, 7, -6,\bul,1, -4, 5) $ &  $ (4, -7, 6, -5, -10, 1, -2, 9, -8, 3) $ \\
 \hline
  $9_{9}$ & 9 & 9 & $ (1, -2, 3, -8, 5, -6, 7, -4, 9) $ &  $ - $ \\
 \hline
  $9_{10}$ & 9 & 9 & $ (1, -4, 3, -2, 5, -8, 7, -6, 9) $ &  $ - $ \\
 \hline
  $9_{11}$ & 9 & 10 & $ (3, -4, 5, -6, 9,\bul, -2, 1,\bul, -8, 7) $ &  $ (4, -5, 6, -7, 10, 1, -2, -9, 8, -3) $ \\
 \hline
  $9_{12}$ & 9 & 10 & $ (3, -4, 9,\bul, -2, 1,\bul, -8, 7, -6, 5) $ &  $ (4, -5, -10, 1, -2, 9, -8, 7, -6, 3) $ \\
 \hline
  $9_{13}$ & 9 & 9 & $ (1, -6, 3, -4, 5, -2, 9, -8, 7) $ &  $ - $ \\
 \hline
  $9_{14}$ & 9 & 11 & $ (3, -2,\bul, 9, -4, 1,\bul, -8, 7, -6, 5) $ &  $ (1, 8, -3, -6, 5, -4, -7, 2, 11, -10, 9) $ \\
 \hline
  $9_{15}$ & 9 & 10 & $ (3, -4, 9, -8, 7,\bul, -2, 1,\bul, -6, 5) $ &  $ (4, -5, 10, -9, 8, 1, -2, -7, 6, -3) $ \\
 \hline
  $9_{16}$ & 9 & 9 & $(1, -8, 5, -6, 7, -2, 3, -4, 9)$ &  $ - $ \\
 \hline
  $9_{17}$ & 9 & 11 & $ (3, -2,\bul, 9, -4, 5, -6, 1,\bul, -8, 7) $ &  $ (1, -2, 3, 6, -5, 4, 7, 10, -9, 8, 11) $ \\
 \hline
  $9_{18}$ & 9 & 9 & $ (1, -8, 3, -6, 5, -4, 7, -2, 9) $ &  $ - $ \\
 \hline
  $9_{19}$ & 9 & 11 & $ (3, -2,\bul, 9, -8, 7, -4, 1,\bul, -6, 5) $ &  $ (1, -8, 3, 6, -5, 4, 7, -2, -11, 10, -9) $ \\
 \hline
  $9_{20}$ & 9 & 10 & $ (5, -4, \bul,1, -2, 7, -8, 9, -6, 3) $ &  $ (4, -5, -10, 1, -2, 9, -6, 7, -8, 3) $ \\
 \hline
  $9_{21}$ & 9 & 10 & $ (5, -2, 1, -8, 7, -6, 9,\bul, -4, 3) $ &  $ (4, -7, 6, -5, 10, 1, -2, -9, 8, -3) $ \\
 \hline
  $9_{22}$ & 9 & 11 & $  (3, -2,\bul, 7, -8, 9, -4, 1,\bul, -6, 5)$ &  $ (1, 10, -5, -8, 7, -6, -9, 2, -3, 4, 11) $ \\
 \hline
  $9_{23}$ & 9 & 9 & $ (3, -4, 9, -8, 5, -2, 1, -6, 7) $ &  $ - $ \\
 \hline
  $9_{24}$ & 9 & 10 & $ (5, -6, 7, -4, \bul, 1, -2, 9, -8, 3) $ &  $(4, -5, 6, -7, -10, 1, -2, 9, -8, 3)  $ \\
 \hline
  $9_{25}$ & 9 & 10 & $ (5, -4,\bul, 9, -8, 1, -2, 7, -6, 3) $ &  $ (4, 7, -6, 5, 10, -1, 2, -9, 8, -3) $ \\
 \hline
  $9_{26}$ & 9 & 11 & $ (3, -4, 5, -2,\bul, 9, -6, 1,\bul, -8, 7) $ &  $ (1, -2, 3, 10, -5, -8, 7, -6, -9, 4, 11) $ \\
 \hline
  $9_{27}$ & 9 & 11 & $ (3, -2,\bul, 9, -4, 5, -8,\bul, 1, -6, 7) $ &  $ (1, 4, -3, 2, 5, -10, -7, 8, -9, -6, 11) $ \\
 \hline
  $9_{28}$ & 9 & 10 & $ (5, -6, 9,\bul, -4,\bul, 1, -2,\bul, 7, -8,\bul, 3) $ &  $ (4, -5, 10, -1, 2, -9, -6, 7, -8, -3) $ \\
 \hline
  $9_{29}$ & 11 & 12 & $ (5, -4,\bul, -9, 8, 1, -2, -7, -10, 11, -6, 3) $ &  $ (4, -7, -10, 11, 6, -5, 12, 1, -2, 9, -8, 3) $ \\
 \hline
  $9_{30}$ & 9 & 11 & $ (5, -6, 9, -2, 1,\bul, -4, 7, -8, 3) $ &  $ (1, -2, 3, -8, 7, 4, -11, 10, -5, 6, -9) $ \\
 \hline
  $9_{31}$ & 9 & 11 & $ (3, -4, 9, \bul, -2, 5, -8,\bul, 1, -6, 7) $ &  $ (1, -10, -3, 8, 5, -6, 7, 4, -9, -2, 11) $ \\
 \hline
  $9_{32}$ & 10 & 11 & $ (2, 7, -6, 5, 8, -1,\bul, 10, -3, 4, -9) $ &  $ (1, -4, 5, 10, -9, -6, 3, -2, -7, 8, 11) $ \\
 \hline
  $9_{33}$ & 11 & 11 & $ (1, 4, -5, 6, 9, -10, 3, -2, 11, 8, -7) $ &  $ - $ \\
 \hline
  $9_{34}$ & 11 & 12 & $ (7, -4, 1, 10, -9, 2, -3, -8, 11,\bul, -6, 5) $ &  $ (4, 9, -8, -5, 12, 1, -2, -11, -6, 7, 10, -3) $ \\
 \hline
  $9_{35}$ & 9 & 9 & $ (3, -2, 1, -6, 5, -4, 9, -8, 7) $ &  $ - $ \\
 \hline
  $9_{36}$ & 9 & 10 & $  (5, -6, 7, -2, 1, -8, 9, \bul, -4, 3)$ &  $ (4, -5, 8, -9, 10, 1, -2, -7, 6, -3) $ \\
 \hline
  $9_{37}$ & 9 & 11 & $ (3, -2, 9, -6, 5, -4, 1, -8, 7) $ &  $ (1, -4, 3, -2, 5, 8, -7, 6, -11, 10, -9) $ \\
 \hline
  $9_{38}$ & 9 & 9 & $ (1, -6, 5, -2, 9, -8, 3, -4, 7)  $ &  $ - $ \\
 \hline
  $9_{39}$ & 9 & 11 & $ (3, -2,\bul, 7, -6,\bul, 1, -4, 9, -8, 5) $ &  $ (3, -2, 1, 8, -7, -4, 11, -10, 5, -6, -9) $ \\
 \hline
  $9_{40}$ & 12 & 12 & $ (6, -7, -10, 3, -2, -11, 12, -1, 4, 9, -8, -5) $ &  $ - $ \\
 \hline
  $9_{41}$ & 11 & 11 & $ (1, -6, 5, -4, -9, 10, -3, 2, 11, -8, 7) $ &  $ - $ \\
 \hline
  $9_{42}$ & 9 & 9 & $ (1, -6, 5, -2, -9, 8, -3, 4, -7) $ &  $ - $ \\
 \hline
  $9_{43}$ & 9 & 10 & $ (5, -6, 7, 2, -1, -8, 9,\bul, 4, -3) $ &  $ (4, -5, 8, -9, 10, -1, 2, 7, -6, 3) $ \\
 \hline
  $9_{44}$ & 9 & 9 & $ (1, 6, -5, -2, 9, -8, -3, 4, -7) $ &  $ - $ \\
 \hline
  $9_{45}$ & 9 & 9 & $ (1, -6, 5, -2, -9, 8, 3, -4, 7) $ &  $ - $ \\
 \hline
  $9_{46}$ & 9 & 9 & $ (3, -2, 1, -6, 5, -4, -9, 8, -7) $ &  $ - $ \\
 \hline
  $9_{47}$ & 10 & 11 & $ (2, 7, -6, -3, 10, \bul, 1, -8, -5, 4, 9) $ &  $ (1, -4, 5, 10, -9, 6, -3, 2, 7, -8, 11) $ \\
 \hline
  $9_{48}$ & 9 & 9 & $ (1, -6, 5, -2, 9, -8, -3, 4, 7) $ &  $ - $ \\
 \hline
  $9_{49}$ & 9 & 11 & $ (3, -2, \bul, -7, 6, \bul, 1, -4, 9, -8, 5) $ &  $ (1, 8, 5, -4, 9, -10, 3, -6, 7, 2, 11) $ \\
 \hline
  $10_{1}$ & 10 & 11  & $ (2, -1,\bul, 10, -9, 8, -7, 6, -5, 4, -3) $ &  $ (1, 4, -3, 2, 11, -10, 9, -8, 7, -6, 5) $ \\
 \hline
  $10_{2}$ & 10 & 11 & $ (2, -1,\bul , 4, -5, 6, -7, 8, -9, 10, -3) $ &  $ (1, -2, 3, -4, 5, -6, 7, 10, -9, 8, 11) $ \\
 \hline
  $10_{3}$ & 10 & 11 & $ (4, -3, 2, -1,\bul, 10, -9, 8, -7, 6, -5) $ &  $ (1, -4, 3, -2, -11, 10, -9, 8, -7, 6, -5) $ \\
 \hline
  $10_{4}$ & 10 & 11 & $ (2, -3, 4, -1, \bul, 10, -9, 8, -7, 6, -5) $ &  $ (1, 6, -3, 4, -5, 2, 11, -10, 9, -8, 7) $ \\
 \hline
  $10_{5}$ & 10 & 11 & $ (2, -3, 4, -5, 6, -7, 10, \bul, -1, 8, -9) $ &  $ (1, -2, 3, -4, 5, -10, -7, 8, -9, -6, 11) $ \\
 \hline
  $10_{6}$ & 10 & 11 & $ (2, -1,\bul, 6, -7, 8, -9, 10, -5, 4, -3) $ &  $(1, -8, 3, -4, 5, -6, 7, -2, -11, 10, -9) $ \\
 \hline
  $10_{7}$ & 10 & 11 & $ (2, -1, \bul, 4, -9, 8, -7, 6, -5, 10, -3) $ &  $ (1, 4, -3, 2, 5, -10, 9, -8, 7, -6, 11) $ \\
 \hline
  $10_{8}$ & 10 & 11 & $ (2, -3, 4, -5, 6, -1,\bul, 10, -9, 8, -7) $ &  $ (1, 8, -3, 4, -5, 6, -7, 2, 11, -10, 9) $ \\
 \hline
  $10_{9}$ & 10 & 11 & $ (2, -3, 4, -1,\bul, 6, -7, 8, -9, 10, -5) $ &  $ (1, -2, 3, -4, 5, 10, -7, 8, -9, 6, 11) $ \\
 \hline
  $10_{10}$ & 10 & 11 & $ (2, -7, 6, -5, 4, -3, 10, \bul, -1, 8, -9) $ &  $ (1, -6, -3, 4, -5, -2, 11, -10, 9, -8, 7) $ \\
 \hline
  $10_{11}$ & 10 & 11 & $ (4, -3, 2, -1,\bul, 8, -9, 10, -7, 6, -5) $ &  $ (1, -6, 3, -4, 5, -2, -11, 10, -9, 8, -7) $ \\
 \hline
  $10_{12}$ & 10 & 11 & $ (2, -3, 4, -5, 10, -9, 8, \bul, -1, 6, -7) $ &  $ (1, -2, 3, -10, -7, 8, -9, -6, 5, -4, 11) $ \\
 \hline
  $10_{13}$ & 10 & 12 & $ (4, -3,\bul ,10, -9, \bul,2, -1,\bul, 8, -7, 6, -5) $ &  $ (4, 7, -6, 5, 12, 1, -2, -11, 10, -9, 8, -3) $ \\
 \hline
  $10_{14}$ & 10 & 11 & $ (2, -3, 4, -5, 10, -9, 6, -1,\bul, 8, -7) $ &  $ (1, -2, 3, -10, 5, 8, -7, 6, 9, -4, 11) $ \\
 \hline
  $10_{15}$ & 10 & 11 & $ (2, -3, 10, -9, 8, \bul, -1, 4, -5, 6, -7) $ &  $ (1, -10, -5, 6, -7, 8, -9, -4, 3, -2, 11) $ \\
 \hline
  $10_{16}$ & 10 & 11 & $ (2, -5, 4, -3, 6, -1,\bul, 10, -9, 8, -7) $ &  $ (1, 8, -3, 6, -5, 4, -7, 2, 11, -10, 9) $ \\
 \hline
  $10_{17}$ & 10 & 11 & $ (2, -3, 4, -5, 10,\bul, -1, 6, -7, 8, -9) $ &  $ (1, -2, 3, -10, -5, 6, -7, 8, -9, -4, 11) $ \\
 \hline
  $10_{18}$ & 10 & 11 & $ (2, -3, 10, -9, 4, -1,\bul, 8, -7, 6, -5) $ &  $ (1, -10, 3, 8, -7, 6, -5, 4, 9, -2, 11) $ \\
 \hline
  $10_{19}$ & 10 &11 & $ (2, -3, 4, -5, 10,\bul, -1, 6, -9, 8, -7) $ &  $ (1, -2, 3, -10, -5, 8, -7, 6, -9, -4, 11) $ \\
 \hline
  $10_{20}$ & 10 & 11 & $ (2, -1, \bul, 8, -9, 10, -7, 6, -5, 4, -3) $ &  $ (1, 4, -3, 2, 9, -10, 11, -8, 7, -6, 5) $ \\
 \hline
  $10_{21}$ & 10 & 11 & $ (2, -1, \bul, 4, -5, 6, -9, 8, -7, 10, -3) $ &  $ (1, -2, 3, -6, 5, -4, 7, 10, -9, 8, 11) $ \\
 \hline
  $10_{22}$ & 10 & 11 & $ (2, -3, 4, -1, \bul, 8, -9, 10, -7, 6, -5) $ &  $ (1, -6, 3, -4, 5, -2, -9, 10, -11, 8, -7) $ \\
 \hline
  $10_{23}$ & 10 & 11 & $ (2, -3, 4, -7, 6, -5, 10, \bul, -1, 8, -9) $ &  $ (1, -4, 3, -2, 5, -10, -7, 8, -9, -6, 11)  $ \\
 \hline
  $10_{24}$ & 10 & 11 & $ (2, -1,\bul, 6, -9, 8, -7, 10, -5, 4, -3) $ &  $ (1, -8, 3, -6, 5, -4, 7, -2, -11, 10, -9) $ \\
 \hline
  $10_{25}$ & 10 & 11 & $ (2, -1, \bul, 4, -9, 6, -7, 8, -5, 10, -3) $ &  $ (1, 4, -3, 2, 5, -10, 7, -8, 9, -6, 11) $ \\
 \hline
  $10_{26}$ & 10 & 11 & $ (2, -3, 4, -1, \bul, 6, -9, 8, -7, 10, -5) $ &  $ (1, -2, 3, 10, -5, 8, -7, 6, -9, 4, 11) $ \\
 \hline
  $10_{27}$ & 10 & 11 & $ (2, -7, 4, -5, 6, -3, 10,\bul, -1, 8, -9) $ &  $ (1, -10, -3, 8, -5, 6, -7, 4, -9, -2, 11) $ \\
 \hline
  $10_{28}$ & 10 & 11 & $ (2, -5, 4, -3, 10, -9, 8,\bul, -1, 6, -7) $ &  $ (1, -4, 3, -2, -9, 10, -11, -8, 7, -6, 5) $ \\
 \hline
  $10_{29}$ & 10 & 12 & $(6, -5, 2, -1, \bul, 8, -9, 10, -7, 4, -3) $ &  $ (4, 9, -6, 7, -8, 5, 12, 1, -2, -11, 10, -3) $ \\
 \hline
  $10_{30}$ & 10 & 11 & $ (2, -5, 4, -3, 10, -9, 6, -1, \bul, 8, -7) $ &  $ (1, -8, 3, 6, -5, 4, 7, -2, 11, -10, 9) $ \\
 \hline
  $10_{31}$ & 10 & 11 & $ (2, -3, 10, -9, 8, \bul, -1, 4, -7, 6, -5) $ &  $ (1, -10, -5, 8, -7, 6, -9, -4, 3, -2, 11) $ \\
 \hline
  $10_{32}$ & 10 & 11 & $ (2, -3, 10, -9, 4, -1, \bul, 6, -7, 8, -5) $ &  $ (1, -10, 3, 8, -5, 6, -7, 4, 9, -2, 11) $ \\
 \hline
  $10_{33}$ & 10 & 11 & $ (2, -5, 4, -3, 10, \bul, -1, 6, -9, 8, -7) $ &  $ (1, -8, -3, 6, -5, 4, -7, -2, 11, -10, 9) $ \\
 \hline
  $10_{34}$ & 10 & 11 & $ (2, -3, 10, -9, 8, -7, 6, \bul, -1, 4, -5) $ &  $ (1, -10, -7, 8, -9, -6, 5, -4, 3, -2, 11) $ \\
 \hline
  $10_{35}$ & 10 & 12 & $ (8, -7, 2, -1, \bul, 10, -9, 6, -5, 4, -3) $ &  $ (4, 7, -6, 5, 12, -11, 10, 1, -2, -9, 8, -3) $ \\
 \hline
  $10_{36}$ & 10 & 11 & $ (2, -3, 10, -9, 8, -7, 4, -1, \bul, 6, -5) $ &  $ (1, -10, 5, 8, -7, 6, 9, -4, 3, -2, 11) $ \\
 \hline
  $10_{37}$ & 10 & 11 & $ (4, -5, 10, -9, 8, \bul, -3, 2, -1, 6, -7) $ &  $ (3, -4, 11, -10, 9, -8, -5, 2, -1, 6, -7) $ \\
 \hline
  $10_{38}$ & 10 & 11 & $ (4, -5, 10, -9, 6, -3, 2, -1, \bul, 8, -7) $ &  $ (3, 6, -5, 4, 11, -10, 7, -2, 1, -8, 9) $ \\
 \hline
  $10_{39}$ & 10 & 11 & $ (2, -3, 10, -9, 4, -5, 6, -1, \bul, 8, -7)  $ &  $ (1, -10, 3, -4, 5, 8, -7, 6, 9, -2, 11) $ \\
 \hline
  $10_{40}$ & 10 & 11 & $ (2, -3, 10, -9, 4, -5, 8, \bul, -1, 6, -7) $ &  $ (1, -10, 3, -8, -5, 6, -7, -4, 9, -2, 11) $ \\
 \hline
  $10_{41}$ & 10 & 12 & $ (4, -3, \bul, 10, -9, \bul, 2, -5, 6, -1, \bul, 8, -7) $ &  $ (4, -5, 6, 9, -8, 7, 12, 1, -2, -11, 10, -3) $ \\
 \hline
  $10_{42}$ & 10 & 12 & $ (4, -3, \bul, 10, -9, \bul, 2, -5, 8, -1, \bul, 6, -7) $ &  $ (4, -9, -6, 7, -8, -5, 12, 1, -2, -11, 10, -3) $ \\
 \hline
  $10_{43}$ & 10 & 12 & $ (4, -5, 10,\bul, -3, 2, \bul, -9, 8, \bul, -1, 6, -7) $ &  $ (4, -5, 12, 1, -2, -11, -8, 9, -10, -7, 6, -3) $ \\
 \hline
  $10_{44}$ & 10 & 12 & $ (4, -5, 10, \bul, -3, 2, \bul, -9, 6, \bul, -1, 8, -7) $ &  $ (4, -5, -12, 1, -2, 11, -6, -9, 8, -7, -10, 3) $ \\
 \hline
  $10_{45}$ & 10 & 13 & $ (4, -3, \bul, 10, -5, 2, \bul, -9, 6, -1, \bul, 8, -7)  $ &  $ (1, 4, -3, 2, 5, 12, -7, -10, 9, -8, -11, 6, 13) $ \\
 \hline
  $10_{46}$ & 10 & 11 & $ (2, -1, \bul, 6, -7, 8, -9, 10, -3, 4, -5) $ &  $ (1, 10, -5, 6, -7, 8, -9, 2, -3, 4, 11) $ \\
 \hline
  $10_{47}$ & 10 & 11 & $ (2, -3, 6, -7, 8, -9, 10, \bul, -1, 4, -5) $ &  $ (1, -10, 5, -6, 7, -8, 9, 2, -3, 4, 11) $ \\
 \hline
  $10_{48}$ & 10 & 11 & $ (2, -3, 8, -9, 10,\bul, -1, 4, -5, 6, -7) $ &  $ (1, -10, -5, 6, -7, 8, -9, -2, 3, -4, 11) $ \\
 \hline
  $10_{49}$ & 10 & 10 & $ (4, -5, 6, -7, 10, -1, 2, -9, 8, -3) $ &  $ - $ \\
 \hline
  $10_{50}$ & 10 & 11 & $ (2, -1, \bul, 6, -9, 8, -7, 10, -3, 4, -5) $ &  $ (1, 10, -5, 8, -7, 6, -9, 2, -3, 4, 11) $ \\
 \hline
  $10_{51}$ & 10 & 11 & $ (2, -3, 6, -9, 8, -7, 10, \bul, -1, 4, -5) $ &  $ (1, -10, 5, -8, 7, -6, 9, 2, -3, 4, 11) $ \\
 \hline
  $10_{52}$ & 10 & 11 & $ (2, -3, 8, -9, 10,\bul, -1, 4, -7, 6, -5) $ &  $ (1, -10, -5, 8, -7, 6, -9, -2, 3, -4, 11) $ \\
 \hline
  $10_{53}$ & 10 & 10 & $ (4, -7, 6, -5, 10, -1, 2, -9, 8, -3) $ &  $  -$ \\
 \hline
  $10_{54}$ & 10 & 11 & $ (4, -5, 8, -9, 10,\bul, -3, 2, -1, 6, -7) $ &  $ (3, -4, 9, -10, 11, -8, -5, 2, -1, 6, -7) $ \\
 \hline
  $10_{55}$ & 10 & 10 & $ (4, -5, 10, -9, 8, -1, 2, -7, 6, -3) $ &  $ - $ \\
 \hline
  $10_{56}$ & 10 & 11 & $ (4, -3, \bul, 8, -9, 10, -5, 2, -1, 6, -7) $ &  $ (3, -4, -9, 10, -11, 8, -5, 2, -1, 6, -7) $ \\
 \hline
  $10_{57}$ & 10 & 11 & $ (6, -7, 10, \bul, -1, 2, -5, 8, -9, 4, -3) $ &  $ (3, -4, -9, 10, -11, -8, 5, -2, 1, -6, 7) $ \\
 \hline
  $10_{58}$ & 10 & 12 & $ (6, -5, \bul, 10, -9, \bul, 2, -1, \bul, 8, -7, 4, -3) $ &  $ (6, 9, -8, 7, 12, 3, -2, 1, 4, -11, 10, -5) $ \\
 \hline
  $10_{59}$ & 10 & 12 & $ (6, -5, \bul, 10, -9, 2, -1, \bul, 4, -7, 8, -3)  $ &  $ (4, -5, 8, 11, -10, 9, 12, 1, -2, -7, 6, -3) $ \\
 \hline
  $10_{60}$ & 10 & 13 & $ (6, -5, \bul, 10, -7, \bul, 2, -1, \bul, 4, \bul, -9, 8, \bul, -3)  $ &  $ (1, 12, -7, -10, 9, -8, -11, 2, -3, -6, 5, -4, 13) $ \\
 \hline
  $10_{61}$ & 10 & 11 & $ (4, -3, 2, -1, \bul, 8, -9, 10, -5, 6, -7)$ &  $ (1, -4, 3, -2, -9, 10, -11, 8, -5, 6, -7) $ \\
 \hline
  $10_{62}$ & 10 & 11 & $ (2, -3, 4, -5, 8, -9, 10, \bul, -1, 6, -7) $ &  $ (1, -2, 3, -10, -7, 8, -9, -4, 5, -6, 11) $ \\
 \hline
  $10_{63}$ & 10 & 10 & $ (4, -5, 10, -1, 2, -9, 8, -7, 6, -3) $ &  $ - $ \\
 \hline
  $10_{64}$ & 10 & 11 & $ (2, -3, 4, -1,\bul, 8, -9, 10, -5, 6, -7) $ &  $ (1, -2, 3, 10, -7, 8, -9, 4, -5, 6, 11) $ \\
 \hline
  $10_{65}$ & 10 & 11 & $ (2, -5, 4, -3, 8, -9, 10, \bul, -1, 6, -7) $ &  $ (1, -4, 3, -2, -9, 10, -11, -8, 5, -6, 7) $ \\
 \hline
  $10_{66}$ & 10 & 10 & $ (4, -5, 10, -1, 2, -9, 6, -7, 8, -3) $ &  $ - $ \\
 \hline
  $10_{67}$ & 10 & 11 & $ (2, -3, 10, -9, 6, -5, 4, -1,\bul, 8, -7) $ &  $ (1, -10, 3, -6, 5, -4, -9, 8, -7, -2, 11) $ \\
 \hline
  $10_{68}$ & 10 & 11 & $ (2, -3, 10, \bul, -1, 6, -5, 4, -9, 8, -7) $ &  $ (1, -10, -3, 6, -5, 4, -9, 8, -7, -2, 11) $ \\
 \hline
  $10_{69}$ & 11 & 13 & $ (3, -2, 11, 6, -5, 4, 7, -10, 1, -8, 9)  $ &  $ (1, -12, -3, -6, 5, -4, -7, -10, 9, -8, -11, -2, 13) $ \\
 \hline
  $10_{70}$ & 10 & 12 & $ (6, -7, 8, -5, 2, -1, \bul, 10, -9, 4, -3) $ &  $ (4, 7, -6, 5, 12, -1, 2, -11, -8, 9, -10, -3) $ \\
 \hline
  $10_{71}$ & 10 & 12 & $ (6, -7, 10, \bul, -5, 2, -1, 8, -9, \bul, 4, -3) $ &  $ (4, 7, -6, 5, 12, -1, 2, -11, -8, 9, -10, -3) $ \\
 \hline
  $10_{72}$ & 10 & 11 & $ (2, -3, 8, -9, 10, -7, 4, -1, \bul, 6, -5) $ &  $ (1, -10, 5, 8, -7, 6, 9, -2, 3, -4, 11) $ \\
 \hline
  $10_{73}$ & 10 & 12 & $ (6, -5, \bul, 10, -7, 4, \bul, -1, 2, -9, 8, -3) $ &  $ (4, -5, 6, 9, -8, 7, -12, 1, -2, 11, -10, 3) $ \\
 \hline
  $10_{74}$ & 10 & 11 & $ (2, -1,\bul, 4, -7, 6, -5, 10, -9, 8, -3) $ &  $ (1, -4, 3, -2, 5, -8, 7, -6, -11, 10, -9) $ \\
 \hline
  $10_{75}$ & 11 & 13  & $ (3, -2, \bul, 11, -4, 5, 8, -7, 6, 1, \bul, -10, 9) $ &  $ (1, 4, -3, 2, 5, 8, -7, 6, 9, 12, -11, 10, 13)  $ \\
 \hline
  $10_{76}$ & 10 & 11 & $ (2, -1,\bul, 8, -9, 10, -7, 4, -5, 6, -3) $ &  $ (1, 4, -3, 2, 9, -10, 11, -8, 5, -6, 7) $ \\
 \hline
  $10_{77}$ & 10 & 11 & $(2, -3, 8, -9, 10, -7, 6, -1, \bul, 4, -5)  $ &  $ (1, -10, -7, 8, -9, -6, 3, -4, 5, -2, 11) $ \\
 \hline
  $10_{78}$ & 10 & 12 & $ (6, -7, 10, \bul, -5, 4,\bul, -1, 2, -9, 8, -3) $ &  $ (4, -9, -6, 7, -8, -5, -12, 1, -2, 11, -10, 3) $ \\
 \hline
  $10_{79}$ & 10 & 12 & $ (4, -5, 8, -9, 10, \bul, -1, 2, -3, 6, -7) $ &  $ (4, -5, 10, -11, 12, 1, -2, -9, -6, 7, -8, 3) $ \\
 \hline
  $10_{80}$ & 10 & 10 & $ (4, -5, 8, -9, 10, -1, 2, -7, 6, -3) $ &  $  - $ \\
 \hline
  $10_{81}$ & 10 & 12 & $ (8, -5, 4, -9, 10,\bul, -1, 2, -7, 6, -3) $ &  $ (4, 7, -6, 5, -10, 11, -12, -1, 2, -9, 8, -3) $ \\
 \hline
  $10_{82}$ & 10 & 12 & $ (2, -5, 6, -7, 8, -1, \bul, 10, -3, 4, -9) $ &  $ (4, -5, -8, 9, -10, 11, -12, -1, 2, 7, -6, -3) $ \\
 \hline
  $10_{83}$ & 10 & 12 & $ (2, -7, 6, -5, 8, -1, \bul, 10, -3, 4, -9) $ &  $ (4, -5, -8, 11, -10, 9, -12, -1, 2, 7, -6, -3) $ \\
 \hline
  $10_{84}$ & 12 & 12 & $ (4, -5, -8, 9, 12, -1, 2, 11, -10, -7, 6, -3) $ &  $ - $ \\
 \hline
  $10_{85}$ & 12 & 12 & $ (4, -5, -8, 9, -10, 11, -12, 1, -2, -7, 6, 3) $ &  $ - $ \\
 \hline
  $10_{86}$ & 12 & 12 & $ (4, -5, -8, 11, -10, 9, -12, 1, -2, -7, 6, 3) $ &  $ - $ \\
 \hline
  $10_{87}$ & 12 & 12 & $ (4, -5, 12, 9, -8, 1, -2, 7, 10, -11, -6, 3) $ &  $ - $ \\
 \hline
  $10_{88}$ & 12 & 13 & $ (2, 7, -6, -3, 12, -11, -4, 5, 8, -1, \bul, 10, -9) $ &  $ (1, 6, -5, 4, 7, 12, -11, -8, 3, -2, -9, 10, 13) $ \\
 \hline
  $10_{89}$ & 12 & 13 & $ (2, 7, -6, -3, 12, -11, 4, -5, -8, 1, \bul, -10, 9) $ &  $ (1, -6, 5, -4, -7, 8, 11, -12, 3, -2, 13, 10, -9) $ \\
 \hline
  $10_{90}$ & 10 & 12 & $ (2, -5, 6, -1, \bul, 10, -9, 8, -3, 4, -7) $ &  $ (4, -7, 6, -5, -10, 11, -12, -1, 2, 9, -8, -3) $ \\
 \hline
  $10_{91}$ & 12 & 12 & $(4, -5, 6, -7, -10, 11, -12, 1, -2, -9, 8, 3) $ &  $ - $ \\
 \hline
  $10_{92}$ & 12 & 12 & $ (4, -5, 12, 9, -8, -1, 2, 7, 10, -11, 6, -3) $ &  $  - $ \\
 \hline
  $10_{93}$ & 12 & 12 & $ (4, -7, 6, -5, -10, 11, -12, 1, -2, -9, 8, 3) $ &  $ - $ \\
 \hline
  $10_{94}$ & 10 & 12 & $ (2, -5, 6, -1, \bul, 8, -9, 10, -3, 4, -7) $ &  $ (4, -5, 6, -7, -10, 11, -12, -1, 2, 9, -8, -3) $ \\
 \hline
  $10_{95}$ & 10 & 11 & $ (2, -5, 8, -9, 4, -3, 10, \bul, -1, 6, -7) $ &  $ (1, -8, -5, 6, -7, -2, 11, -10, 3, -4, 9) $ \\
 \hline
  $10_{96}$ & 12 & 14 & $ (6, -5, \bul, -10, 9, -2, 1, \bul,  -4, -7, -12, 11, -8, 3) $ &  $ (6, -7, -14, 13, 10, -9, 8, -3, 2, 11, -12, -1, -4, 5) $ \\
 \hline
  $10_{97}$ & 11 & 11 & $ (1, -4, 7, 10, -9, 8, 3, -2, 11, -6, 5) $ &  $ - $ \\
 \hline
  $10_{98}$ & 10 & 11 & $ (2, -1, \bul, 6, -7, 10, -3, 4, -9, 8, -5) $ &  $ (1, 8, -5, 4, -9, 10, -3, 6, -7, 2, 11) $ \\
 \hline
  $10_{99}$ & 10 & 12 & $  (4, -5, 8, \bul, -1, 2, \bul, -9, 10, \bul -3, 6, -7) $ &  $ (4, -5, 12, 1, -2, -9, 10, -11, -6, 7, -8, 3) $ \\
 \hline
  $10_{100}$ & 12 & 13 & $ (2, -3, 4, 9, -10, -1, \bul, 12, 5, -6, 7, -8, -11) $ &  $ (1, -2, 3, -6, 9, -10, 11, -12, -5, 4, -13, 8, -7) $ \\
 \hline
  $10_{101}$ & 11 & 13 & $ (3, -2,\bul,  -7, 8, -9, -6, \bul, 1, -4, 11, -10, 5) $ &  $(1, 6, -5, 4, -9, 10, -3, 2, -13, 12, -11, 8, -7)$ \\
 \hline
  $10_{102}$ & 10 & 12 & $ (2, -5, 6, -1, \bul, 10, -3, 4, -9, 8, -7) $ &  $ (4, -5, -8, 9, -12, -1, 2, 11, -10, 7, -6, -3) $ \\
 \hline
  $10_{103}$ & 12 & 13 & $ (2, -3, 4, 9, -10, -1, \bul, 12, 7, -6, 5, -8, -11) $ &  $ (1, -2, 3, -6, 9, -12, 11, -10, -5, 4, -13, 8, -7) $ \\
 \hline
  $10_{104}$ & 12 & 13 & $ (2, -3, 4, 7, -8, -1, \bul, 10, -11, 12, 5, -6, -9) $ &  $ (1, -2, 3, -10, -7, 8, -9, -4, 13, -12, -5, 6, -11) $ \\
 \hline
  $10_{105}$ & 12 & 13 & $ (2, -7, 6, -5, -10, 11, -4, 3, 12, \bul, -1, 8, -9) $ &  $ (1, -2, 3, 8, -7, 6, 11, -12, 5, -4, -13, 10, -9) $ \\
 \hline
  $10_{106}$ & 10 & 13 & $ (2, -3, 4, -7, 8, -1, \bul, 10, -5, 6, -9) $ &  $ (1, -2, 3, 10, 7, -8, 9, -4, -13, 12, -5, 6, -11) $ \\
 \hline
  $10_{107}$ & 10 & 13 & $ (2, -3, 10, -7, 6, \bul, -1, 4, -9, 8, -5) $ &  $ (1, -12, -3, 8, -7, -4, 11, -10, -5, 6, -9, -2, 13) $ \\
 \hline
  $10_{108}$ & 12 & 13 & $ (2, -3, 4, 7, -8, -1, \bul, 12, -11, 10, 5, -6, -9) $ &  $ (1, -2, 3, 8, -7, -4, -13, 12, -11, 10, 5, -6, -9) $ \\
 \hline
  $10_{109}$ & 10 & 13 & $(2, -3, 6, -7, 10, \bul, -1, 4, -5, 8, -9)$ &  $ (1, -2, 3, 6, -7, 8, -11, 12, 5, -4, -13, -10, 9) $ \\
 \hline
  $10_{110}$ & 10 & 12 & $ (2, -3, 8, -7, 4, -1, \bul, 10, -9, 6, -5) $ &  $ (4, -5, -12, 9, -8, -1, 2, 7, -10, 11, -6, 3) $ \\
 \hline
  $10_{111}$ & 12 & 12 & $ (4, -5, 8, -9, -12, -1, 2, 11, -10, -7, 6, -3) $ &  $ - $ \\
 \hline
  $10_{112}$ & 12 & 13 & $ (2, -3, 4, -5, -8, 9, 12, \bul, -1, -6, 7, 10, -11) $ &  $ (1, -2, 3, 6, -7, -12, 11, 8, 5, -4, -9, 10, 13)  $ \\
 \hline
  $10_{113}$ & 12 & 12 & $ (4, 9, -8, -5, 12, -1, 2, -11, -6, 7, 10, -3) $ &  $ - $ \\
 \hline
  $10_{114}$ & 12 & 13 & $ (2, -5, 4, -3, -8, 9, 12, \bul, -1, -6, 7, 10, -11) $ &  $ (1, -4, 3, -2, -7, 8, 13, -12, -9, -6, 5, 10, -11) $ \\
 \hline
  $10_{115}$ & 10 & 13 & $ (2, -7, 6, -3, 10, \bul, -1, 8, -5, 4, -9) $ &  $ (1, -6, 5, -4, -7, -12, 11, 8, -3, 2, -9, 10, 13) $ \\
 \hline
  $10_{116}$ & 12 & 13 & $ (2, -3, 4, 7, -8, 1, \bul, -10, 11, -12, -5, 6, 9) $ &  $ (1, -4, 5, -6, -9, 10, -3, 2, 11, -12, 13, 8, -7) $ \\
 \hline
  $10_{117}$ & 12 & 12 & $ (4, -5, -8, 9, -12, 1, -2, 11, -10, -7, 6, 3) $ &  $ - $ \\
 \hline
  $10_{118}$ & 12 & 13 & $ (2, -3, 4, 7, -8, -1, \bul, 12, 5, -6, -9, 10, -11) $ &  $ (1, -2, 3, 6, -7, 8, 11, -12, -5, 4, 13, 10, -9) $ \\
 \hline
  $10_{119}$ & 12 & 13 & $ (2, -1, \bul, 6, -7, -12, 11, -10, 3, -4, -9, 8, -5) $ &  $ (1, 10, -5, 6, 9, -2, 13, -12, 3, 8, -7, 4, 11) $ \\
 \hline
  $10_{120}$ & 13 & 13 & $(1, -8, 7, 4, 11, -12, 3, -2, 13, 10, 5, -6, 9)  $ &  $ - $ \\
 \hline
  $10_{121}$ & 13  & 13 & $ (1, -4, 5, 8, 11, -12, -7, 6, 3, -2, 13, 10, -9) $ &  $ - $ \\
 \hline
  $10_{122}$ & 12 & 13 & $ (2, -3, -6, 7, 12, \bul, -1, 10, -9, -4, 5, 8, -11) $ &  $ (1, 4, -5, -10, 9, 6, 3, -2, 13, -12, -7, 8, 11) $ \\
 \hline
  $10_{123}$ & 12 & 14 & $ (2, 5, -6, -9, 10, -1, \bul, 12, -3, 4, 7, -8, -11) $ &  $ (6, -7, -10, 11, -14, 1, -2, -13, 12, -3, 4, 9, -8, -5) $ \\
 \hline
  $10_{124}$ & 10 & 11 & $ (2, -1,\bul, -6, 7, -8, 9, -10, 3, -4, 5) $ &  $ (1, 10, 5, -6, 7, -8, 9, -2, 3, -4, 11) $ \\
 \hline
  $10_{125}$ & 10 & 10 & $ (4, -5, 6, -7, 10, 1, -2, 9, -8, 3) $ &  $ - $ \\
 \hline
  $10_{126}$ & 10 & 10 & $ (4, -5, 6, -7, 10, 1, -2, -9, 8, 3) $ &  $ - $ \\
 \hline
  $10_{127}$ & 10 & 10 & $ (4, -5, 6, -7, 10, -1, 2, 9, -8, -3) $ &  $ - $ \\
 \hline
  $10_{128}$ & 10 & 11 & $ (2, -1,\bul, -6, 9, -8, 7, -10, 3, -4, 5) $ &  $ (1, 10, 5, -8, 7, -6, 9, -2, 3, -4, 11) $ \\
 \hline
  $10_{129}$ & 10 & 10 & $ (4, -7, 6, -5, 10, 1, -2, 9, -8, 3) $ &  $ - $ \\
 \hline
  $10_{130}$ & 10 & 10 & $ (4, -7, 6, -5, 10, 1, -2, -9, 8, 3) $ &  $ - $ \\
 \hline
  $10_{131}$ & 10 & 10 & $ (4, -7, 6, -5, 10, -1, 2, 9, -8, -3) $ &  $ - $ \\
 \hline
  $10_{132}$ & 10 & 10 & $ (4, -5, 10, -9, 8, 1, -2, -7, 6, 3) $ &  $ - $ \\
 \hline
  $10_{133}$ & 10 & 10 & $(4, -5, 10, -9, 8, -1, 2, 7, -6, -3) $ &  $ - $ \\
 \hline
  $10_{134}$ & 10 & 11 & $ (4, -3, \bul, -8, 9, -10, 5, -2, 1, -6, 7) $ &  $ (3, -4, 9, -10, 11, 8, 5, -2, 1, -6, 7) $ \\
 \hline
  $10_{135}$ & 10 & 10 & $ (4, -5, 10, -9, 8, 1, -2, 7, -6, 3) $ &  $ - $ \\
 \hline
  $10_{136}$ & 10 & 11 & $ (2, 7, -6, -3, 10, \bul -1, -8, 5, -4, -9) $ &  $ (1, 4, -5, -10, 9, -6, 3, -2, -7, 8, 11) $ \\
 \hline
  $10_{137}$ & 10 & 11 & $ (2, -7, 6, -3, -10, \bul, 1, 8, -5, 4, 9) $ &  $ (1, -6, 5, -4, -9, 10, 3, -2, 11, 8, -7) $ \\
 \hline
  $10_{138}$ & 10 & 12 & $ (6, -5, \bul, 10, -9, -2, 1, \bul, -4, 7, -8, 3) $ &  $ (4, -5, 8, 11, -10, 9, 12, -1, 2, 7, -6, 3) $ \\
 \hline
  $10_{139}$ & 10 & 11 & $ (2, -3, 4, -1, \bul, -8, 9, -10, 5, -6, 7) $ &  $ (1, -2, 3, 10, 7, -8, 9, -4, 5, -6, 11) $ \\
 \hline
  $10_{140}$ & 10 & 10 & $ (4, -5, 10, 1, -2, -9, 8, -7, 6, 3) $ &  $ - $ \\
 \hline
  $10_{141}$ & 10 & 10 & $ (4, -5, 10, -1, 2, 9, -6, 7, -8, -3) $ &  $ -  $ \\
 \hline
  $10_{142}$ & 10 & 11 & $ (2, -5, 4, -3, 8, -9, 10, \bul, 1, -6, 7) $ &  $ (1, -4, 3, -2, 9, -10, 11, 8, 5, -6, 7) $ \\
 \hline
  $10_{143}$ & 10 & 10 & $ (4, -5, 10, 1, -2, -9, 6, -7, 8, 3) $ &  $ - $ \\
 \hline
  $10_{144}$ & 10 & 10 & $ (4, -5, 10, -1, 2, 9, -8, 7, -6, -3) $ &  $ - $ \\
 \hline
  $10_{145}$ & 10 & 11 & $ (2, -7, 6, -3, -10, \bul, 1, -8, 5, -4, 9) $ &  $ (1, -10, -3, -6, 5, -4, 9, -8, 7, -2, 11) $ \\
 \hline
  $10_{146}$ & 10 & 11 & $ (2, -3, -8, 7, -4, -1, \bul, 10, -9, -6, 5) $ &  $ (1, -10, 3, 6, -5, 4, -9, 8, -7, -2, 11) $ \\
 \hline
  $10_{147}$ & 10 & 11 & $ (2, -1,\bul, -6, 7, 10, -3, 4, -9, 8, -5) $ &  $ (1, -10, -3, -6, 5, -4, -9, 8, -7, -2, 11) $ \\
 \hline
  $10_{148}$ & 10 & 10 & $ (4, -5, 8, -9, 10, 1, -2, -7, 6, 3) $ &  $ - $ \\
 \hline
  $10_{149}$ & 10 & 10 & $ (4, -5, 8, -9, 10, -1, 2, 7, -6, -3) $ &  $ - $ \\
 \hline
  $10_{150}$ & 10 & 10 & $ (4, -5, -8, 9, -10, 1, -2, 7, -6, 3) $ &  $ - $ \\
 \hline
  $10_{151}$ & 10 & 10 & $ (4, -5, -8, 9, -10, -1, 2, -7, 6, -3) $ &  $ - $ \\
 \hline
  $10_{152}$ & 10 & 12 & $ (4, -5, 8, -9, 10, \bul, 1, -2, 3, -6, 7) $ &  $ (4, -5, 10, -11, 12, -1, 2, 9, 6, -7, 8, -3) $ \\
 \hline
  $10_{153}$ & 10 & 10 & $ (4, -5, 8, -9, 10, 1, -2, 7, -6, 3) $ &  $ - $ \\
 \hline
  $10_{154}$ & 10 & 12 & $ (8, -5, 4, -9, 10, \bul, 1, -2, 7, -6, 3) $ &  $ (4, 7, -6, 5, -10, 11, -12, 1, -2, 9, -8, 3) $ \\
 \hline
  $10_{155}$ & 10 & 11 & $ (2, -3, 4, 7, -8, -1, \bul, 10, -5, 6, -9) $ &  $ (1, -2, 3, 8, -7, -4, -11, 10, -5, 6, -9) $ \\
 \hline
  $10_{156}$ & 10 & 11 & $ (2, -1, \bul, 6, -7, -10, 3, -4, 9, -8, -5) $ &  $ (1, -8, -5, 4, -9, 10, -3, 6, -7, -2, 11) $ \\
 \hline
  $10_{157}$ & 10 & 12 & $ (4, -5, -8, \bul, 1, -2,\bul, -9, 10, \bul, -3, -6, 7) $ &  $ (4, -5, 12, 9, -8, -1, 2, -7, 10, -11, 6, 3) $ \\
 \hline
  $10_{158}$ & 10 & 11 & $ (2, -3, -8, 7, 4, -1, \bul, 10, -9, 6, -5) $ &  $ (1, 4, -5, 6, 9, -10, -3, 2, 11, -8, 7) $ \\
 \hline
  $10_{159}$ & 10 & 11 & $ (2, -3, 4, 7, -8, 1, \bul, -10, 5, -6, 9) $ &  $ (1, -4, 5, 10, -9, -6, -3, 2, 7, -8, 11) $ \\
 \hline
  $10_{160}$ & 10 & 11 & $ (2, -7, 6, -3, 10, \bul, 1, -8, -5, 4, 9) $ &  $ (1, -2, 3, 8, -7, -4, 11, -10, 5, -6, -9) $ \\
 \hline
  $10_{161}$ & 10 & 11 & $ (2, -3, 4, -7, 8, 1, \bul, -10, 5, -6, 9) $ &  $ (1, -2, 3, -8, 7, -4, 11, -10, 5, -6, -9) $ \\
 \hline
  $10_{162}$ & 10 & 11 & $ (2, -3, 4, -7, 8, 1, \bul, -10, 5, -6, 9) $ &  $ (1, -2, 3, -8, 7, -4, 11, -10, 5, -6, -9) $ \\
 \hline
  $10_{163}$ & 10 & 11 & $ (2, -5, 6, -1, \bul, 10, 3, -4, -9, 8, -7) $ &  $ (3, -2, 1, 8, -7, -4, -11, 10, -5, 6, -9) $ \\
 \hline
  $10_{164}$ & 10 & 11 & $ (2, -7, 6, -3, 10, \bul, -1, 8, 5, -4, -9) $ &  $ (1, 4, -5, -10, 9, 6, -3, 2, -7, 8, 11) $ \\
 \hline
  $10_{165}$ & 11 & 11 & $ (1, 6, -5, 4, -9, 10, 3, -2, 11, -8, 7) $ &  $ - $ \\
 \hline
\end{longtable}
}




\end{document}